\documentclass[11pt,a4paper,reqno]{amsart}
\usepackage{amsmath,amssymb,amsthm,amsfonts}
\usepackage[export]{adjustbox}
\usepackage{amsbsy}
\usepackage[utf8]{inputenc}
\usepackage[normalem]{ulem}
\usepackage{xspace}
\usepackage{cancel}
\usepackage[greek,english]{babel}
\usepackage{url}
\usepackage{enumerate}
\usepackage{wrapfig}
\usepackage{enumitem}
\usepackage{multicol}
\usepackage{moreenum}
\usepackage{xcolor}
\usepackage{longtable}
\usepackage{array}
\usepackage{aliascnt}
\usepackage{bm}
\usepackage{multirow}

\makeatletter
\newcommand\myitem[1][]{\item[#1]\refstepcounter{dummy}\def\@currentlabel{#1}}
\makeatother
\usepackage{subcaption}
\usepackage{comment}
\usepackage{tikz}
\usetikzlibrary{calc}
\usetikzlibrary{backgrounds}
\usepackage[bottom=2.5cm,top=2.5cm, left=2.5cm, right=2.5cm]{geometry}

\usepackage{multirow}
\usepackage{array,multirow}

\definecolor{LinkColor}{rgb}{0,0,0} 
\usepackage[colorlinks=true,linkcolor=LinkColor,citecolor=LinkColor,urlcolor= LinkColor, naturalnames, hyperindex, pdfstartview=FitH, bookmarksnumbered, plainpages]{hyperref}
\usepackage{cleveref}
\makeatletter
\newcommand{\slunlhd}{%
	\mathrel{\mathpalette\sl@unlhd\relax}%
}

\newcommand{\sl@unlhd}[2]{%
	\sbox\z@{$#1\lhd$}%
	\sbox\tw@{$#1\leqslant$}%
	\dimen@=\ht\tw@
	\advance\dimen@-\ht\z@
	\ifx#1\displaystyle
	\advance\dimen@ .2pt
	\else
	\ifx#1\textstyle
	\advance\dimen@ .2pt
	\fi
	\fi
	\ooalign{\raisebox{\dimen@}{$\m@th#1\lhd$}\cr$\m@th#1\leqslant$\cr}%
}
\makeatother

\newtheorem{innercustomthm}{Theorem}[section]

\crefname{innercustomthm}{Theorem}{Theorems}

\newtheorem{theorem}{Theorem}[section]

\newaliascnt{corollary}{theorem}
\newtheorem{corollary}[corollary]{Corollary}
\aliascntresetthe{corollary}

\newaliascnt{lemma}{theorem}
\newtheorem{lemma}[lemma]{Lemma}
\aliascntresetthe{lemma}

\newaliascnt{proposition}{theorem}
\newtheorem{proposition}[proposition]{Proposition}
\aliascntresetthe{proposition}

\theoremstyle{definition}

\newaliascnt{definition}{theorem}

\aliascntresetthe{definition}

\newaliascnt{remark}{theorem}
\newtheorem{remark}[remark]{Remark}
\aliascntresetthe{remark}

\newaliascnt{example}{theorem}

\aliascntresetthe{example}

\newcommand{\SG}[1]{\text{SG}[#1]}
\newcommand{\SL}{\operatorname{SL}}
\newcommand{\GL}{\operatorname{GL}}

\newcommand{\para}{\para\vspace{.25cm}}
\newcommand{\Aut}{\operatorname{Aut}}
\newcommand{\Out}{\operatorname{Out}}
\newcommand{\Inn}{\operatorname{Inn}}

\newcommand{\C}{\textup{C}}
\newcommand{\F}{\mathbb{F}}

\renewcommand{\O}{\operatorname{O}}
\newcommand{\Frob}{\textup{Fr}}

\newcommand{\N}{\textup{N}}

\newcommand{\fl}[1]{\textup{\bfseries [#1]}}
\renewcommand{\O}{\operatorname{O}}

\newcommand{\qand}{\quad \text{and} \quad}

\author[A. Jindal]{Anu Jindal}
\address{Department of Mathematics, Indian Institute of Technology Roorkee, Roorkee (Uttarakhand)-247667, India.}
\email{\href{mailto:anu_rj@ma.iitr.ac.in}{anu\_rj@ma.iitr.ac.in}}
\author[S. Maheshwary]{Sugandha Maheshwary}
\address{Department of Mathematics, Indian Institute of Technology Roorkee, Roorkee (Uttarakhand)-247667, India.}
\email{\href{mailto:msugandha@ma.iitr.ac.in}{msugandha@ma.iitr.ac.in}}

\thanks{The second author gratefully acknowledges the support by Science  \& Engineering Research Board (SERB),  DST (Department of Science and Technology), India (SRG/2023/000180).}

\date{}

\title{Prime and N-Prime graphs of solvable uniformly semi-rational groups}
\begin{document}
	\maketitle
	
\begin{abstract}
In this paper, we study the prime graph, also known as Gruenberg--Kegel graph, realizable by finite solvable uniformly semi-rational groups. This is primarily achieved by classifying the prime graphs realizable by metanilpotent groups. We also introduce a new class of groups, called MP* groups,
which extends the class of groups with Magnus property and forms a natural subclass of uniformly semi-rational groups. For MP* groups, we give a complete classification of the prime graphs realized by them. In the process, we obtain the realizability of prime graphs by uniformly semi-rational groups, assuming them to be in varied substantial subclasses of solvable groups, namely, $2$-Frobenius, metacyclic, metabelian, nilpotent-by-abelian, abelian-by-cyclic and cyclic-by-abelian. We also classify the so called N-prime graphs for all these classes of groups. The prime graph question for uniformly semi-rational groups has also been addressed.  
\end{abstract}

\section{Introduction}
 Prime graphs were originally introduced in the 1970s to investigate certain cohomological  questions associated with integral representations of finite groups. These graphs are also quite often referred to as GK-graphs, being introduced and initially studied by Karl Gruenberg and Otto Kegel. These graphs also depict various interesting properties of a group and at times determine the group or its type. These graphs have been studied from various  perspectives (\cite{AKK10}, \cite{ZM12}, \cite{BC15}, \cite{GKLNS15}, \cite{GS19},  \cite{CM22},  \cite{GV22}, \cite{BKMdR24}, \cite{DGLdR24}, \cite{KPSZ25}, \cite{KMRRY25}, \cite{DGLdR25}). Quite recently, in \cite{PdRV25}, N-prime graphs of  a group were defined, which are closely related to prime graphs and encode relatively more information of the group.

 Throughout this paper, all groups are assumed to be finite. For a group $G$, we denote its prime graph and N-prime graph by $\Gamma(G)$ and $\Gamma_N(G)$, respectively. The vertices of $\Gamma(G)$ and $\Gamma_N(G)$ are the primes dividing the order of $G$. For a group $G$, $\Gamma(G)$ is the simple undirected graph, in which primes $p$ and $q$ are joined by an edge $p-q$, if and only if $G$ contains an element of order $pq$ whereas, $\Gamma_N(G)$
 is the simple directed graph which has the arc $q \rightarrow p$ if and only if there exists an element $x$ of order $p$ in $G$ such that $\langle x \rangle$ is normalized by an element $y$ of order $q$ in $G$. For an undirected graph $\bm{\Gamma}$ (respectively, a directed graph $\bm{\Gamma_N}$),
 we often say that $G$ realizes $\bm{\Gamma}$ (respectively, $\bm{\Gamma_N}$), equivalently $\bm{\Gamma}$ (respectively, $\bm{\Gamma_N}$) is realized by $G$, if $\Gamma(G)=\bm{\Gamma}$ (respectively, $\Gamma_N{(G)}=\bm{\Gamma_N}$). For a class $C$ of groups, $\bm{\Gamma}$ (respectively,  $\bm{\Gamma_N}$) is said to be realizable by $C$, if there exists a group $G\in C$ such that $\Gamma(G)=\bm{\Gamma}$ (respectively, $\Gamma_N{(G)}=\bm{\Gamma_N}$).
 
  One of the problems that has evolved in last few years is to determine the realizability of a graph as the prime graph of a group belonging to a particular class of groups. The classes considered so far have quite a restricted set of vertices, with at most $3$ or $4$ elements. Hence, apparently, there are finitely many possibilities of such prime graphs and working on a particular class of groups should restrict it even further. But, surprisingly what seems as a naive question, is a non-trivial problem and hard to address completely for a given fixed class. In this article, we proceed in similar direction. We study prime and N-prime graphs realizable by uniformly semi-rational groups, a class of groups which evolved as a generalization of well known class of rational groups. 

Recall that a group $G$ is said to be rational, if all the character values of $G$ are rational and it turns out that for these groups, every element $x\in G$ is such that it satisfies $x^j$ is conjugate to  $x$,  whenever $j$ is relatively prime to the order of $x$. Now, if for a group $G$,  there exists an integer $m$ relatively prime to the exponent of $G$, such that, for every $x\in G$ and every integer $j$ relatively prime to the order of $x$, the element $x^j$ is conjugate to either $x$ or $x^m$, then $G$ is said to be uniformly semi-rational. In particular, if $m=-1$, such groups are called inverse semi-rational, and are popularly known as cut groups because of the fact that for an inverse semi-rational group $G$, the central units of its integral group ring $\mathbb{Z}G$ are trivial. In \cite{BKMdR24}, \cite{DGLdR24} and \cite{DGLdR25}, the authors investigated the prime graphs of solvable rational and inverse semi-rational groups. Continuing further, in this article we consider uniformly rational groups and study prime graphs as well as N-prime graphs for these groups.

The paper is organized as follows: After setting up requisite notation, we collect some \linebreak results and observations on uniformly semi-rational groups, which are used in rest of the paper. Many of these results are generalizations of the analogous results known for inverse semi-rational groups. In Section~$3$, we classify all prime graphs realizable by solvable Frobenius groups and $2$-Frobenius groups. It may be noted that Frobenius and 2-Frobenius groups form an important class of groups in the study of prime graphs, because of an important result which states that if $G$ is solvable then $\Gamma(G)$ is disconnected if and only if $G$ is Frobenius or $2$-Frobenius (see \cite{Wil81}).  In Section~$4$, we classify the prime graphs of metanilpotent uniformly semi-rational groups and its subclasses, such as nilpotent-by-abelian, metabelian, abelian-by-cyclic, cyclic-by-abelian, and metacyclic groups. We classify all prime graphs realizable by metanilpotent uniformly semi-rational groups, with the exception of three remaining graphs. Moreover, two of these three graphs are simultaneously realizable or simultaneously non-realizable. Consequently, the classification is reduced to only two unresolved cases. For all the subclasses stated above, we provide a complete classification of the realizable prime graphs. In Section~$5$, we introduce the class of MP* groups, which extends the class of groups with Magnus property and forms a subclass of uniformly semi-rational groups.  Studying prime graphs of MP* groups naturally brings us a step closer to that of uniformly semi-rational groups. We prove that all MP* groups are metanilpotent and classify all prime graphs realizable by MP* groups. In Section~$6$, we classify the N-prime graphs for the above classes. In Section~$7$, we attempt the prime graph question for uniformly semi-rational groups.

\section*{Notation}

We mostly use the standard notation but state here explicitly for ease of the reader. \\

The symbol $\varphi$ is used to denote Euler's totient function. For intgers $a$ and $b$, $a$ divides $b$ is symbolised as $a\mid b$ and $\operatorname{gcd}(a,b)$ denotes the greatest common divisor of $a$ and $b$. 

For $n\in \mathbb{N}$, we denote by $C_n$ a cyclic group of order $n$. The symmetric and alternating groups on $n$ symbols are denoted by $S_n$ and $A_n$, respectively. The dihedral and quaternion groups of order $n$ are denoted by $D_n$ and $Q_n$, respectively, while $QD_{2^n}$ denotes the quasidihedral group of order $2^n$. As is customary, the letter $p$ will often denote a prime. Furthermore, $\SL(n,p)$ and $\GL(n,p)$ denote the special linear group and the general linear group of degree $n$ over the finite field $\mathbb{F}_p$ of order $p$, respectively. For $q=p^n$, if $V$ denotes $\mathbb{F}_q$-vector space, then $\GL(V)$ is the set of all invertible linear transformations from $V$ to itself. 

The cardinality of a set $X$ is denoted by $|X|$. Throughout, $G$ denotes a finite group whose order and exponent are denoted by $|G|$ and $\operatorname{exp}(G)$ respectively. For $g,h\in G$, we write $|g|$ for the order of $g$, $g^h$ for $h^{-1}gh$ and $g^G$ for conjugacy class of $g$ in $G$. By $H\leq G$ ($H\unlhd G$), we mean that $H$ is a subgroup (normal subgroup) of $G$. The prime spectrum of $G$, that is, the set of primes dividing $|G|$, is denoted by $\pi(G)$, and $\Aut(G)$ denotes the automorphism group of $G$. The center of $G$ is denoted by $\mathcal{Z}(G)$. For a subset $X$ of $G$, the normalizer and the centralizer of $X$ in $G$ are denoted by $N_G(X)$ and $C_G(X)$, respectively. For $p\in \pi(G)$, let $G_p$ denote a Sylow $p$-subgroup of $G$. If $\pi$ is a set of primes, a Hall $\pi$-subgroup  of $G$ and a Hall $\pi'$-subgroup of G will be denoted by  $G_{\pi}$ and $G_{\pi'}$ respectively. By Hall's Theorem; see, for instance, \cite[9.1.7]{Rob82}, every solvable group contains a Hall $\pi$-subgroup for each set of primes $\pi$, and any two Hall $\pi$-subgroups are conjugate in $G$. In particular, they are isomorphic. Thus, for a solvable group $G$, the subgroup $G_\pi$ is unique up to isomorphism.

Every semi-direct product $G\rtimes H$ considered in this paper is assumed not to be a direct product. Moreover, the notation $G=N\rtimes_{\mathrm{Fr}}K$ means that $G$ is a Frobenius group uniquely determined, up to isomorphism, by the isomorphism types of its kernel $N$ and complement $K$. $\SG{n,m}$ denotes the  $m$-th group of order $n$ in the Small Groups Library of \textsf{GAP} \cite{GAP4}.

\section{Preliminaries}
As mentioned above, the aim of this article is to study the prime and the N-prime graphs of uniformly semi-rational groups. For brevity, uniformly semi-rational groups shall be called as USR groups. To determine the possible prime graphs realizable by solvable USR groups, we begin by stating the prime spectrum of such groups.

\begin{proposition}[{\cite[Theorem~4.14]{CM25}}] If $G$ is a finite solvable USR group, then \linebreak $\pi(G) \subseteq \{2, 3, 5, 7, 13\}$.
\end{proposition}

Rather, we have limited choices for the possible prime spectra of solvable USR groups. For this, we consider $B_G(g) := N_G(\langle g \rangle) / C_G(g)$, for an element $g$ of a group  $G$, as defined in \cite{CD10}. The map associating $x \in N_G(\langle g \rangle)$ with the automorphism of $\langle g \rangle$ mapping $g$ to $g^x$ induces an injective homomorphism $\iota_g: B_G(g) \rightarrow \Aut(\langle g \rangle)$. Thus $B_G(g)$ is canonically isomorphic to a subgroup of $\Aut(\langle g \rangle)$. An element $g \in G$ is called semi-rational if each generator of $\langle g \rangle$ is conjugate to $g$ or $g^{m_g}$ in $G$, where $m_g$ is an integer coprime to $|g|$.  If $G$ is a USR group, then there exists an integer $m$ coprime to exponent of $G$ such that, for every $g \in G$, each generator of $\langle g \rangle$ is conjugate to $g$ or $g^m$ in $G$. In this vein, $G$ is $m$-semi-rational and using \cite[Proposition~4.5]{CM25}, we have the following:

	\begin{proposition} \label{equivalentUSR}
The following statements are equivalent for a finite group $G$:
		\begin{itemize}
\item[(i)]$G$ is an $m$-semi-rational group.
\item[(ii)] For every $g\in G$ either $|B_G(g)|=\varphi(|g|)$ or $|B_G(g)|=\frac{\varphi(|g|)}{2}$ and in the latter case $g$ and $g^{m}$ are not conjugate in $G$.
\item[(iii)] For each $g \in G$, $\Aut(\langle g \rangle) \simeq B_G(g)\langle \tau \rangle$ where $\tau(g) = g^m$ and $\tau^2 \in B_G(g)$.

		\end{itemize}
	\end{proposition}
	
On the lines of \cite[Lemma~2.6]{BKMdR24}, we have the following:
		\begin{lemma}\label{PrimePower}
		Let $G$ be a finite group and let $g \in G$ be an element of order $p^n$ or $2p^n$, where $p$ is an odd prime. The element $g$ is  semi-rational in $G$, if and only if $B_G(g)$ contains an element of order $p^{n-1}(\frac{p-1}{2})$. 
	\end{lemma}	
	
It thus follows that if $G$ has an element $g$ of order $5, 7$ or $13$ which is semi-rational, then $B_G(g)$ and hence $G$ has an element of order $2, 3$ and $6$ respectively. Consequently, the possibilities of prime spectra of USR groups are restricted.
	\begin{lemma}\label{PrimeSpectrum}
	If $G$ is a finite solvable USR group, then $\pi(G)$ is one of the following sets: $\{2\}, \{3\}, \{2, 3\}, \{2, 5\}, \{3, 7\}, \{2, 3, 5\}, \{2, 3, 7\}, \{2, 3, 13\}, \{2, 3, 5, 7\}, \{2, 3, 5, 13\},
	\{2, 3, 7, 13\}, \linebreak \{2, 3, 5, 7, 13\}$.
\end{lemma}
Whether or not there exist USR groups for each prime spectrum listed above cannot be \linebreak commented at the moment. However, on the lines of \cite[Lemma~3.2]{BKMdR24} and using \Cref{equivalentUSR}, we obtain the following crucial result on the edges of the prime graph of a finite USR group.

	\begin{theorem}\label{Edges}
		Let $\bm{\bm{\Gamma}}$ be the prime graph of a finite USR group $G$.
		\begin{itemize}
			\item[(i)] If $2 - 7 \in \bm{\bm{\Gamma}}, ~3 - 5 \in \bm{\Gamma},~ 3 - 7 \in \bm{\Gamma} $,~ $5 - 7 \in \bm{\Gamma}$ or $13 \in \pi(G)$, then $2 - 3 \in \bm{\Gamma}$.
			\item[(ii)] If $5 - 7 \in \bm{\Gamma}$, then $2 - 7 \in \bm{\Gamma}$ and $3 - 5 \in \bm{\Gamma}$.
			\item[(iii)] If $5 - 13 \in \bm{\Gamma}$, then $2 - 5 \in \bm{\Gamma},~3 - 5 \in \bm{\Gamma}$ and $2 - 13 \in \bm{\Gamma}$.
			\item[(iv)] If $7 - 13 \in \bm{\Gamma},$ then $2 - 7\in \bm{\Gamma}, ~3 - 7 \in \bm{\Gamma}$ and $3 - 13 \in \bm{\Gamma}$.
			
		\end{itemize}
	\end{theorem}

Further, on the lines of \cite[Lemma~3.4]{BKMdR24}, we obtain that the exponent of an abelian Sylow $p$-subgroup of a finite USR group divides $p$ or $4$. Using this, we obtain the following generalization of \cite[Lemma~3.5]{BKMdR24}.

	\begin{proposition}\label{CyclicSylow}
Let $G$ be a finite USR group and let $G_p$ be a normal Sylow $p$-subgroup of $G$. If $q$ belongs to $\pi(G)$ is such that $p \neq q$ and $G$ does not contain an element of order $pq$, then Sylow subgroup $G_q$ is one of the following:
\begin{itemize}
	\item[(i)] a cyclic group of order dividing $4$ or $q$,
	\item[(ii)] a quaternion group of order $8$ or $16$.
	
\end{itemize} 
	\end{proposition}

The following corollary immediately follows from above proposition and \cite[Lemma~4.3]{BKMdR24}. 
	
	\begin{corollary}\label{G2G3Exclusions}
		Let $G$ be a finite USR group. 
		\begin{enumerate}
			\item[(i)]\label{Not4pG2Cyclic} If $G_2$ is cyclic and $p$ is an odd prime, then $G$ has no elements of order $4p$.
			\item[(ii)]\label{No21} If  $G_3$ is cyclic, then $G$ has no elements of order $3\cdot 7$ and $3\cdot 13$.
		\end{enumerate}  
	\end{corollary}
	
In order to conclude about the prime graphs of solvable USR groups, we shall make use of various results which we list now. To begin with, we have Lucido's three prime lemma which states that if $X$ is a set containing $3$ vertices of the prime graph $\bm{\Gamma}$ of a finite solvable group, then at least two elements of $X$ are joined by an edge in $\bm{\Gamma}$ \cite[Proposition~1]{Luc99}.

For solvable groups, there is a crucial criterion on connected components of its prime graph, due to Gruenberg and Kegel \cite[Theorem A]{Wil81}. To understand the criterion, we need the notion of $2$-Frobenius groups.

 Recall that the Fitting subgroup of $G$, denoted by $F(G)$, is the unique largest nilpotent normal subgroup of $G$. Equivalently, it is the direct product of the largest normal $p$-subgroups $\O_p(G)$ of $G$.

The Fitting series 
\[1 = F_0(G) \leqslant F_1(G) \leqslant F_2(G) \leqslant ... \] of $G$ 
is defined by 
$$F_0(G) = 1 \qand \frac{F_{j}(G)}{F_{j-1}(G)}=F\left(\frac{G}{F_{j-1}(G)}\right) \text{ for } j\geqslant 1.$$ 
Observe that if $G$ is a non-trivial solvable group, then $F(G)\ne 1$. The least integer $n$ for which the Fitting series reaches $G$ is called the Fitting length of $G$, and is denoted by $\ell_F(G)$.
A finite group $G$ is said to be $2$-Frobenius if there exist normal subgroups $N \leqslant K$ of $G$, such that $K$ is a Frobenius group with kernel $N$, and $G/N$ is a Frobenius group with kernel $K/N$. Equivalently, $G$ is $2$-Frobenius if and only if $F_2(G)$ and $G/F_1(G)$ are Frobenius groups. In this case, $F_1(G)$ is called the lower kernel, while $G/F_2(G)$ is called the upper complement of $G$. 
Moreover, $F_2(G)/F_1(G)$ is called the upper kernel as well as the lower complement.	

\begin{proposition}[{\cite[Theorem~2.4]{BKMdR24}}]\label{ConnectedComponents} Let $G$ be a finite solvable group, then $\Gamma(G)$ has at most $2$ connected components, and has exactly $2$ components, if and only if $G$ is a Frobenius group or a  $2$-Frobenius group.
\end{proposition}

\begin{lemma}[{\cite[Lemma~2.1]{GKLNS15}}]\label{GKLNS}
	Let $G$ be a finite $2$-Frobenius group. Then $G/ F_2(G)$ is cyclic, $F_2(G) / F_1(G)$ is cyclic of odd order, and $F_1(G)$ is not a cyclic group.
\end{lemma}
We thus obtain constraint on the number of primes dividing Frobenius and $2$-Frobenius USR groups.

\begin{theorem}\label{Frobenius3Vertices} 
The order of a finite USR group which is a solvable Frobenius group or a $2$-Frobenius group, is divisible by at most $3$ primes.
\end{theorem}

\begin{proof} If $G$ is a solvable Frobenius USR group, then the result follows directly from \cite[Main theorem]{PV26}. Suppose that $G$ is a $2$-Frobenius USR group with $|\pi(G)|>3$. Since every $2$-Frobenius group is solvable, \Cref{PrimeSpectrum} implies that $\pi(G)=\{2,3,5,7\}$, $\{2,3,5,13\}$, $\{2,3,7,13\}$ or $\{2,3,5,7,13\}.$ By \Cref{GKLNS}, $G/F(G)$ is a Frobenius group with cyclic kernel of odd order and cyclic complement. Using again \cite[Main theorem]{PV26}, we have that $G/F(G)$ is isomorphic to one of the following Frobenius groups: 
	\[C_3 \rtimes C_2,\quad C_5 \rtimes C_4,\quad C_5 \rtimes C_2, \quad C_7 \rtimes C_3,\quad C_7\rtimes C_6 \quad \text{or} \quad C_{13} \rtimes C_6. \]
	\begin{description}
		\item[Case I] $\underline{ \pi(G) = \{2, 3, 5, 7\}}$

\noindent Here, $G/F(G)$ is isomorphic to one of the following: \[C_3 \rtimes C_2,\quad C_5 \rtimes C_4,\quad C_5 \rtimes C_2, \quad C_7 \rtimes C_3 \quad  \text{or} \quad C_7\rtimes C_6. \] 
If $G/F(G) \simeq C_3 \rtimes C_2$, then $\mathcal{Z}(F(G))$ contains an element $x$ of order $35$.  Now, \linebreak $B_G(x) = N_G(\langle x \rangle) / C_G(x)$ has to be isomorphic to a subgroup of quotient of $G/F(G)$. Hence $[\operatorname{Aut}(\langle x\rangle) : B_G(x)]>2$, a contradiction.
If $G/F(G) \simeq C_5 \rtimes C_4$ or $C_5 \rtimes C_2$, then center of $F(G)$ contains an element of order $7$, say $x$. Now, $B_G(x) = \N_G(\langle x \rangle) / C_G(x)$ has to be isomorphic to a subgroup of quotient of $G/F(G)$. Since $|\Aut(\langle x\rangle)|=6$, we get that $[\Aut(\langle x \rangle) : B_G(x)] > 2$, again a contradiction.
Similarly, if $G/F(G) \simeq C_7 \rtimes C_3$, then there is an element $x$ of order $5$ in the center of $F(G)$, yielding contradiction. Now, if $G/F(G) \simeq  C_7 \rtimes C_6$, then by Lucido's three prime lemma and \Cref{Edges}, it follows that the prime graph of $G$ contains a triangle in which either the vertex $5$ or the vertex $7$ is isolated. If the vertex $5$ is isolated, then $G$ contains an element of order $21$, but $G/F(G)$ does not contains an element of order $21$. Therefore, we may assume that the vertex $7$ is isolated and choose $G$ to be of minimal order which realizes this graph. Then there exists an element $x \in G$ of order $15$. Consequently, $4 \mid |B_G(x)|$, which implies $4 \mid |G|$ and hence $2 \mid |F(G)|$. It follows that $G/O_2(G)$ is a USR group whose prime graph is same as that of $G$, yielding a contradiction.

\item[Case II] $\underline{ \pi(G) = \{2, 3, 5, 13\}}$ 

\noindent In this case, $G/F(G) $ is isomorphic to one of the following: 
	\[C_3 \rtimes C_2,\quad C_5 \rtimes C_4,\quad C_5 \rtimes C_2 \quad \text{or} \quad C_{13} \rtimes C_6. \]
Proceeding as in Case I, one may check that $G/F(G)$ is not isomorphic to $C_3 \rtimes C_2$, $C_5 \rtimes C_2$ and $C_5 \rtimes C_4$. Also, if $G/F(G) \simeq C_{13} \rtimes C_6$, then by Lucido's three prime lemma and \Cref{Edges}, it follows that $\Gamma(G)$ contains a triangle in which either the vertex $5$ or the vertex $13$ is isolated.  If the vertex $5$ is isolated, then as $5$ divides $|F(G)|$ and $G$ does not have an element of order $10$. So, we get that $G_2 \simeq C_2$, which is not possible as $G$ contains an element of order $13 \cdot 3$. Therefore, we assume that the vertex $13$ is isolated and choose $G$ to be a group of minimal order which realizes this graph. Then, there exists an element $x \in G$ of order $15$. Consequently, $4 \mid |B_G(x)|$ and $2 \mid |F(G)|$. It follows that $G/O_2(G)$ is a USR group whose prime graph is same as that of $G$, yielding a contradiction. 

\item[Case III] $\underline{ \pi(G) = \{2, 3, 7, 13\}}$

\noindent In this case, $G/F(G) $ is isomorphic to one of the following: 
\[C_3 \rtimes C_2,\quad C_7 \rtimes C_3,\quad C_7 \rtimes C_6 \quad \text{or} \quad C_{13} \rtimes C_6. \] If $G/F(G) \simeq C_3 \rtimes C_2,$ or  $C_7 \rtimes C_3$, then arguing as in Case I, we get a contradiction. Now, if $G/F(G) \simeq  C_7 \rtimes C_6$, then as observed in Case II, $\Gamma(G)$ contains a triangle in which either the vertex $7$ or the vertex $13$ is isolated. Now, if the vertex $7$ is isolated, then using \Cref{G2G3Exclusions}, the center of $F(G)$ contains an element $x$ of order $3 \cdot 13$. Therefore, $12 \mid |B_G(x)|$, which is not possible as $B_G(x)$ has to be isomorphic to a subgroup of a quotient of $G/F(G)$.

Therefore, we may assume that the vertex $13$ is isolated and choose $G$ to be a group of minimal order which realizes this graph. By \Cref{G2G3Exclusions}, $3 \mid |F(G)|$. Hence, $G/O_3(G)$ is a USR group realizing the same prime graph as $G$, a contradiction. Also, if $G/F(G) \simeq C_{13} \rtimes C_6$, then it is not possible as in the former case.
\item[Case IV] $\underline{ \pi(G) = \{2, 3, 5, 7, 13\}}$

\noindent In this case, $G/F(G) $ is isomorphic to one of the following: 
	\[C_3 \rtimes C_2,\quad C_5 \rtimes C_4,\quad C_5 \rtimes C_2, \quad C_7 \rtimes C_3,\quad C_7\rtimes C_6 \quad \text{or} \quad C_{13} \rtimes C_6. \] 
Note that each of these possibilities imply presence of an element of non-admissible order in the center of $F(G)$.
	\end{description}
	
	\end{proof}
\begin{corollary}
	The prime graph of a finite solvable USR group with more than $3$ vertices is connected.
\end{corollary}

\begin{figure}[htbp]
	\centering
	\scriptsize
	\setlength{\tabcolsep}{2pt}
	\renewcommand{\arraystretch}{1.15}
	
	\begin{adjustbox}{max width=\textwidth}
		\begin{tabular}{|c|c|}
			\hline
			$\pi(G)$&Prime graphs\\
			\hline

			$\{2\}$  &
		(a) \begin{subfigure}{.15\textwidth}
				\centering
				\begin{tikzpicture}
					\draw (0.5,1.5) node{$C_2$};
					\node[label=east:{$2$}] at (0.5,1) (3){};
					\foreach \p in {3}{
						\draw[fill=black] (\p) circle (0.075cm);
					}
				\end{tikzpicture}
			\end{subfigure} \\
			\hline
			$\{3\}$&
			(b) \begin{subfigure}{.15\textwidth}
				\centering
				\begin{tikzpicture}
					\draw (0.5,1.5) node{$C_3$};
					\node[label=east:{$3$}] at (0.5,1) (3){};
					\foreach \p in {3}{
						\draw[fill=black] (\p) circle (0.075cm);
					}
				\end{tikzpicture}
			\end{subfigure} \\
			\hline
			
			$\{2, 3\}$ & (c)
			\begin{subfigure}{.18\textwidth}
				\centering
				\begin{tikzpicture}
					\draw (0.5,1.5) node{$S_3=C_3\rtimes_\Frob C_2$};
					\node[label=west:{$2$}] at (0,1) (2){};
					\node[label=east:{$3$}] at (0.5,1) (3){};
					\foreach \p in {2,3}{
						\draw[fill=black] (\p) circle (0.075cm);
					}
				\end{tikzpicture}
			\end{subfigure}
			(d) \hspace{-0.4cm} 
			\begin{subfigure}{.18\textwidth}
				\centering
				\begin{tikzpicture}
					\draw (0.5,1.5) node{$S_3 \times C_2$};
					\node[label=west:{$2$}] at (0,1) (2){};
					\node[label=east:{$3$}] at (0.5,1) (3){};
					\foreach \p in {2,3}{
						\draw[fill=black] (\p) circle (0.075cm);
					}
					\draw (2)--(3);
				\end{tikzpicture}
			\end{subfigure} \\
			\hline
			$\{2, 5\}$ &
			(e)\hspace{-0.4cm} 
			\begin{subfigure}{.18\textwidth}
				\centering
				\begin{tikzpicture}
					\draw (0.5,1.5) node{$C_5^2 \rtimes_\Frob Q_8$};				
					\node[label=west:{$2$}] at (0.5,1) (2){};
					\node[label=west:{$5$}] at (0.5,0.5) (3){};
					\foreach \p in {2,3}{
						\draw[fill=black] (\p) circle (0.075cm);
					}
				\end{tikzpicture}
			\end{subfigure}
			\hspace{-0.4cm} (f)\hspace{-0.1cm} 
			\begin{subfigure}{.20\textwidth}
				\centering
				
				\begin{tikzpicture}
					\draw (0.5,1.5) node{$[C_5^2 \rtimes_\Frob Q_8] \times C_2$};								
					\node[label=west:{$2$}] at (0.5,1) (2){};
					\node[label=west:{$5$}] at (0.5,0.5) (5){};
					\foreach \p in {2,5}{
						\draw[fill=black] (\p) circle (0.075cm);
					}
					\draw (2)--(5);
				\end{tikzpicture}
			\end{subfigure}\\
			\hline
			$\{3, 7\}$ &
			(g) \hspace{-0.5cm} 
			\begin{subfigure}{.17\textwidth}
				\centering
				\begin{tikzpicture}
					\draw (0.5,1.5) node{$C_7 \rtimes_\Frob C_3$};
					\node[label=east:{$3$}] at (0.5,1) (3){};
					\node[label=east:{$7$}] at (0.5,0.5) (7){};
					\foreach \p in {3,7}{
						\draw[fill=black] (\p) circle (0.075cm);
					}
				\end{tikzpicture}
			\end{subfigure}\\
			\hline
			$\{2, 3, 5\}$ &
			(h) \hspace{-0.3cm} 			
			\begin{subfigure}{.2\textwidth}
				\centering
				
				\begin{tikzpicture}
					\draw (0.5,1) node{$C_5^4 \rtimes_\Frob (C_3 \times Q_8)$};				
					\node[label=west:{$2$}] at (0,0.5) (2){};
					\node[label=east:{$3$}] at (0.5,0.5) (3){};
					\node[label=west:{$5$}] at (0,0) (5){};
					\foreach \p in {2,3,5}{
						\draw[fill=black] (\p) circle (0.075cm);
					}
					\draw (2)--(3);
				\end{tikzpicture}
			\end{subfigure}
			\hspace{-0.3cm} (i) \hspace{-0.1cm}
			\begin{subfigure}{.22\textwidth}
				\centering
				
				\begin{tikzpicture}
					\draw (0.5,1) node{$[C_5^4 \rtimes_\Frob (C_3 \times Q_8)] \times C_2$};								
					\node[label=west:{$2$}] at (0,0.5) (2){};
					\node[label=east:{$3$}] at (0.5,0.5) (3){};
					\node[label=west:{$5$}] at (0,0) (5){};
					\foreach \p in {2,3,5}{
						\draw[fill=black] (\p) circle (0.075cm);
					}
					\draw (2)--(3);
					\draw (2)--(5);
				\end{tikzpicture}
			\end{subfigure}
			\hspace{0.3cm} (j) \hspace{-0.4cm}
			\begin{subfigure}{.22\textwidth}
				\centering
				\begin{tikzpicture}
					\draw (0.5,1) node{$\text{SG}[60,7] = C_{15} \rtimes C_4$};				
					\node[label=west:{$2$}] at (0,0.5) (2){};
					\node[label=east:{$3$}] at (0.5,0.5) (3){};
					\node[label=west:{$5$}] at (0,0) (5){};
					\foreach \p in {2,3,5}{
						\draw[fill=black] (\p) circle (0.075cm);
					}
					\draw (2)--(3);
					\draw (3)--(5);
				\end{tikzpicture}
			\end{subfigure}
			\hspace{-0.5cm} (k) \hspace{-0.2cm}
			\begin{subfigure}{.22\textwidth}
				\centering
				
				\begin{tikzpicture}
					\draw (0.5,1) node{$[C_{15} \rtimes C_4] \times S_3$};								
					\node[label=west:{$2$}] at (0,0.5) (2){};
					\node[label=east:{$3$}] at (0.5,0.5) (3){};
					\node[label=west:{$5$}] at (0,0) (5){};
					\foreach \p in {2,3,5}{
						\draw[fill=black] (\p) circle (0.075cm);
					}
					\draw (2)--(3);
					\draw (2)--(5);
					\draw (3)--(5);
				\end{tikzpicture}
			\end{subfigure}\\
			\hline
			$\{2, 3, 7\}$&
			(l) \hspace{-0.5cm}
			\begin{subfigure}{.2\textwidth}
				\centering
				\begin{tikzpicture}
					\draw (0.5,1) node{$C_{7} \rtimes_\Frob  C_6$};				
					\node[label=west:{$2$}] at (0,0.5) (2){};
					\node[label=east:{$3$}] at (0.5,0.5) (3){};
					\node[label=east:{$7$}] at (0.5,0) (7){};
					\foreach \p in {2,3,7}{
						\draw[fill=black] (\p) circle (0.075cm);
					}
					\draw (2)--(3);
				\end{tikzpicture}
			\end{subfigure}
			\hspace{-0.2cm} (m) \hspace{-0.5cm}
			\begin{subfigure}{.22\textwidth}
				\centering
				\begin{tikzpicture}
					\draw (0.5,1) node{$[C_{7} \rtimes_\Frob  C_6] \times C_2$};								
					\node[label=west:{$2$}] at (0,0.5) (2){};
					\node[label=east:{$3$}] at (0.5,0.5) (3){};
					\node[label=east:{$7$}] at (0.5,0) (7){};
					\foreach \p in {2,3,7}{
						\draw[fill=black] (\p) circle (0.075cm);
					}
					\draw (2)--(3);
					\draw (2)--(7);
				\end{tikzpicture}
			\end{subfigure}
			\hspace{-0.2cm} (n) \hspace{-0.5cm}
			\begin{subfigure}{.22\textwidth}
				\centering
				\begin{tikzpicture}
					\draw (0.5,1) node{$[C_{7} \rtimes_\Frob  C_6] \times C_3$};				
					\node[label=west:{$2$}] at (0,0.5) (2){};
					\node[label=east:{$3$}] at (0.5,0.5) (3){};
					\node[label=east:{$7$}] at (0.5,0) (7){};
					\foreach \p in {2,3,7}{
						\draw[fill=black] (\p) circle (0.075cm);
					}
					\draw (2)--(3);
					\draw (3)--(7);
				\end{tikzpicture}
			\end{subfigure}
			\hspace{-0.2cm} (o) \hspace{-0.3cm}
			\begin{subfigure}{.2\textwidth}
				\centering
				\begin{tikzpicture}
					\draw (0.5,1) node{$[C_{7} \rtimes_\Frob  C_3] \times S_3$};				
					\node[label=west:{$2$}] at (0,0.5) (2){};
					\node[label=east:{$3$}] at (0.5,0.5) (3){};
					\node[label=east:{$7$}] at (0.5,0) (7){};
					\foreach \p in {2,3,7}{
						\draw[fill=black] (\p) circle (0.075cm);
					}
					\draw (2)--(3);
					\draw (2)--(7);
					\draw (3)--(7);
				\end{tikzpicture}
			\end{subfigure}\\
			\hline
			$\{2, 3, 13\}$ &
			(p) \hspace{-0.5cm}

			\begin{subfigure}{.28\textwidth}
				\centering
				\begin{tikzpicture}
					\draw (0.5,1) node{$[C_{13} \rtimes_\Frob C_6]$};				
					\node[label=west:{$2$}] at (0,0.5) (2){};
					\node[label=east:{$3$}] at (0.5, 0.5) (3){};
					
					\node[label=east:{$13$}] at (0.75, 0) (13){};
					\foreach \p in {2,3,13}{
						\draw[fill=black] (\p) circle (0.075cm);
					}
					\draw (2)--(3);
					
				\end{tikzpicture}
			\end{subfigure}
			\hspace{-0.2cm} (q) \hspace{-0.3cm}
			\begin{subfigure}{.28\textwidth}
				\centering
				\begin{tikzpicture}
					\draw (0.5,1) node{$[C_{13} \rtimes_\Frob C_6]  \times C_2$};				
					\node[label=west:{$2$}] at (0,0.5) (2){};
					\node[label=east:{$3$}] at (0.5, 0.5) (3){};
					
					\node[label=east:{$13$}] at (0.75, 0) (13){};
					\foreach \p in {2,3,13}{
						\draw[fill=black] (\p) circle (0.075cm);
					}
					\draw (2)--(3);
					\draw (2)--(13);
					
				\end{tikzpicture}
			\end{subfigure} 
			\hspace{0.4cm} (r) \hspace{-0.4cm}
			\begin{subfigure}{.28\textwidth}
				\centering
				\begin{tikzpicture}
					\draw (0.5,1) node{$[C_{13} \rtimes_\Frob C_6] \times S_3$};					
					\node[label=west:{$2$}] at (0,0.5) (2){};
					\node[label=east:{$3$}] at (0.5, 0.5) (3){};
					\node[label=east:{$13$}] at (0.75, 0) (13){};
					\foreach \p in {2,3,13}{
						\draw[fill=black] (\p) circle (0.075cm);
					}
					\draw (2)--(3);
					\draw (2)--(13);
					\draw (3)--(13);
					
				\end{tikzpicture}
			\end{subfigure}\\ [0.3cm]
		
			& (x*) \hspace{-0.4cm}
			\begin{subfigure}{.28\textwidth}
				\centering
				\begin{tikzpicture}
					
					\node[label=west:{$2$}] at (0,0.5) (2){};
					\node[label=east:{$3$}] at (0.5, 0.5) (3){};
					\node[label=east:{$13$}] at (0.75, 0) (13){};
					\foreach \p in {2,3,13}{
						\draw[fill=black] (\p) circle (0.075cm);
					}
					\draw (2)--(3);
					\draw (3)--(13);
					
				\end{tikzpicture}
			\end{subfigure}\\
			\hline

			$\{2, 3, 5, 7\}$& (s) \hspace{-0.3cm}			
			\begin{subfigure}{.28\textwidth}
				\centering
				\begin{tikzpicture}
					\node[align=center, font=\scriptsize, text width=4cm] at (0.5,1) {$[C_5^2 \rtimes_\Frob Q_8] \times [C_{7} \rtimes_\Frob  C_3]$};				
					\node[label=west:{$2$}] at (0,0.5) (2){};
					\node[label=east:{$3$}] at (0.5, 0.5) (3){};
					\node[label=west:{$5$}] at (0, 0) (5){};
					\node[label=east:{$7$}] at (0.5, 0) (7){};
					\foreach \p in {2,3,5,7}{
						\draw[fill=black] (\p) circle (0.075cm);
					}
					\draw (2)--(3);
					\draw (2)--(7);;
					\draw (3)--(5);
					\draw (5)--(7);
				\end{tikzpicture}
			\end{subfigure}
			\hspace{-0.2cm} (t) \hspace{-0.3cm}
			\begin{subfigure}{.28\textwidth}
				\centering
				\begin{tikzpicture}
					\node[align=center, font=\scriptsize, text width=4cm] at (0.5,1){$[C_5^2 \rtimes_\Frob Q_8] \times [C_{7} \rtimes_\Frob  C_3] \times C_2$};				
					\node[label=west:{$2$}] at (0,0.5) (2){};
					\node[label=east:{$3$}] at (0.5, 0.5) (3){};
					\node[label=west:{$5$}] at (0, 0) (5){};
					\node[label=east:{$7$}] at (0.5, 0) (7){};
					\foreach \p in {2,3,5,7}{
						\draw[fill=black] (\p) circle (0.075cm);
					}
					\draw (2)--(3);
					\draw (2)--(5);
					\draw (2)--(7);
					\draw (3)--(5);
					\draw (5)--(7);
				\end{tikzpicture}
			\end{subfigure} 
			\hspace{0.4cm} (u) \hspace{-0.4cm}
			\begin{subfigure}{.28\textwidth}
				\centering
				\begin{tikzpicture}
					\node[align=center, font=\scriptsize, text width=4cm] at (0.5,1){$[C_5^2 \rtimes_\Frob Q_8] \times [C_{7} \rtimes_\Frob  C_6] \times S_3$};					
					\node[label=west:{$2$}] at (0,0.5) (2){};
					\node[label=east:{$3$}] at (0.5, 0.5) (3){};
					\node[label=west:{$5$}] at (0, 0) (5){};
					\node[label=east:{$7$}] at (0.5, 0) (7){};
					\foreach \p in {2,3,5,7}{
						\draw[fill=black] (\p) circle (0.075cm);
					}
					\draw (2)--(3);
					\draw (2)--(5);
					\draw (2)--(7);
					\draw (3)--(5);
					\draw (3)--(7);
					\draw (5)--(7);
				\end{tikzpicture}
			\end{subfigure}\\
			\hline		
			$\{2, 3, 5, 13\}$& (v) \hspace{-0.5cm}
			\begin{subfigure}{.45\textwidth} 
				\centering
				\begin{tikzpicture}
					\node[align=center, font=\scriptsize, text width=4cm] at (0.35,1.2) {$[C_5^2 \rtimes_{\Frob} Q_8] \times [C_{13} \rtimes_{\Frob} C_6]$};				
					\node[label=west:{$2$}] at (0,0.5) (2){};
					\node[label=east:{$3$}] at (0.7, 0.5) (3){};
					\node[label=west:{$5$}] at (0, 0) (5){};
					\node[label=east:{$13$}] at (0.7, 0) (13){};
					
					\foreach \p in {2,3,5,13}{
						\draw[fill=black] (\p) circle (0.075cm);
					}
					
					\draw (2)--(3);
					\draw (2)--(13);
					\draw (3)--(5);
					\draw (5)--(13);
					\draw (2)--(5);
				\end{tikzpicture}
			\end{subfigure}
			\hfill 
			(w) \hspace{-0.3cm}
			\begin{subfigure}{.45\textwidth}
				\centering
				\begin{tikzpicture}
					\node[align=center, font=\scriptsize, text width=4cm] at (0.35,1.2) {$[C_5^2 \rtimes_{\Frob} Q_8] \times [C_{13} \rtimes_{\Frob} C_6] \times S_3$};				
					
					\node[label=west:{$2$}] at (0,0.5) (2){};
					\node[label=east:{$3$}] at (0.7, 0.5) (3){};
					\node[label=west:{$5$}] at (0, 0) (5){};
					\node[label=east:{$13$}] at (0.7, 0) (13){};
					
					\foreach \p in {2,3,5,13}{
						\draw[fill=black] (\p) circle (0.075cm);
					}
					
					\draw (2)--(3);
					\draw (2)--(5);
					\draw (2)--(13);
					\draw (3)--(5);
					\draw (5)--(13);
					\draw (3)--(13);
				\end{tikzpicture}
			\end{subfigure}\\
			\hline		
			
			$\{2, 3, 7, 13\}$& (y*) \hspace{-0.5cm}
			\begin{subfigure}{.45\textwidth} 
				\centering
				\begin{tikzpicture}
					\node[label=west:{$2$}] at (0,0.5) (2){};
					\node[label=east:{$3$}] at (0.7, 0.5) (3){};
					\node[label=west:{$7$}] at (0, 0) (7){};
					\node[label=east:{$13$}] at (0.7, 0) (13){};
					
					\foreach \p in {2,3,7,13}{
						\draw[fill=black] (\p) circle (0.075cm);
					}
					
					\draw (2)--(3);
					\draw (2)--(13);
					\draw (3)--(7);
					\draw (7)--(13);
					\draw (2)--(7);
					\draw (3)--(13);
				\end{tikzpicture}
			\end{subfigure}\\
			\hline
			
			$\{2, 3, 5, 7, 13\}$ & 
			(z*) \hspace{-0.5cm}
			\begin{subfigure}{.45\textwidth} 
				\centering
				\begin{tikzpicture}
					
					\node[label=west:{$2$}] at (0,1.2) (2){};
					\node[label=east:{$3$}] at (1.4, 1.2) (3){};
					\node[label=west:{$5$}] at (0, 0) (5){};
					\node[label=east:{$7$}] at (1.4, 0) (7){};
					\node[label=east:{$13$}] at (2.8,0.6) (13){};
					
					\foreach \p in {2,3,5,7,13}{
						\draw[fill=black] (\p) circle (0.075cm);
					}
					
					\draw (2)--(3);
					\draw (2)--(5);
					\draw (2)--(7);
					\draw (3)--(5);
					\draw (5)--(7);
					\draw (3)--(7);		
					\draw(2)--(13);
					\draw(3)--(13);
					\draw(5)--(13);
					\draw(7)--(13);

				\end{tikzpicture}
			\end{subfigure}\\
			\hline
				\end{tabular}
	\end{adjustbox}

\caption{\label{Examplesprime} Possible prime graphs of metanilpotent USR groups.}		
\end{figure}

 \Cref{Examplesprime} and \Cref{fig:Examples} respectively contain the prime graphs and the N-prime graphs, possibly realizable by metanilpotent USR groups. The groups listed along graphs in these figures are the examples of metanilpotent USR groups realizing those graphs. For the undecided graphs, no groups are stated.

Note that in the forthcoming sections, we shall be dealing with various subclasses of solvable groups, to check the realizability of prime graphs and N-prime graphs in those classes. For a class $C$, any graph realizable by $C$, which is not in \Cref{Examplesprime} or \Cref{fig:Examples}, shall be taken up separately. Also, if the example stated in \Cref{Examplesprime} or \Cref{fig:Examples} does not belong to $C$, then groups in $C$ with that graph shall be listed explicitly. It may be said that \Cref{Examplesprime} and \Cref{fig:Examples} are used as common references for all the results in this article and anything additionally needed is presented as and when needed.

\begin{center}
		\begin{tabular}{|c|c|c|c|}
		\hline
		&Class $C$ of groups& Undirected graph $\bm{\Gamma}$ & $G \in C$ such that $\Gamma(G) = \bm{\Gamma}$ \\
		
		\hline		
		\multirow{2}{*}{(I)}
		& \multirow{2}{*}{Frobenius }
		&  (c), (e), (g), (h), (l), (p) & as in  \Cref{Examplesprime}\\
		& & $\Gamma_1 = (2-5 ~~3)$ &$C_3^4 \rtimes_{Fr} \langle x, y~ |~ x^5 = y^4 = 1, x^y = x^{-1} \rangle$   \\
	
		\hline
		\multirow{6}{*}{(II)}
		& \multirow{6}{*}{$2$-Frobenius }
		& (c) & $C_2^2 \rtimes (C_3 \rtimes_{Fr} C_2)$\\
		
		& & (e) & $C_2^4 \rtimes (C_5 \rtimes_{Fr} C_4)$\\
		
		& & (g) & $C_3^6 \rtimes (C_7 \rtimes_{Fr} C_3)$\\
		
		& & (h) & $C_3^4 \rtimes (C_5 \rtimes_{Fr} C_2)$\\
		
		& & (l) & $C_2^6 \rtimes (C_7 \rtimes_{Fr} C_6)$\\
		
		& & (p) & $C_2^{12} \rtimes (C_{13} \rtimes_{Fr} C_6)$\\
		\hline
		
		\multirow{5}{*}{(III)}
		& \multirow{5}{*}{Nilpotent-by-abelian }
		& (a)-(d), (g), (j)-(r) & as in \Cref{Examplesprime}\\
		&& (e) & $C_5 \rtimes_{Fr} C_4$\\
		& & (f) & $(C_5 \rtimes _{Fr} C_4) \times C_2$\\
		&& (h)& $C_5^2 \rtimes_{Fr} C_6$\\
		&& (i) & $(C_5^2 \rtimes_{Fr} C_6) \times C_2$   \\
		\hline

		(IV)&Metabelian  & (a)-(r) & as in (III)   \\
		\hline
		
		\multirow{2}{*}{(V)}
		& \multirow{2}{*}{Abelian-by-cyclic  }
		& (a)-(j), (l)-(q) & as in (III) \\
		& & (k) & $(C_{15} \rtimes C_4) \times C_2$  $=$ SG$[60,7] \times C_2$ \\
		\hline
		\multirow{2}{*}{(VI)}
		& \multirow{2}{*}{Cyclic-by-abelian  }
		& (a)-(g),  (j), (l)-(r) & as in (III)\\
		& & (k)& as in (V)\\
		\hline
		(VII)&Metacyclic  & (a)-(g), (j)-(q)& as in (VI)\\
		\hline
		(VIII)& MP*  & (a)-(g), (j)-(r) & as in (III)\\
		
		\hline
	\end{tabular}
\captionof{table}{Groups realizing prime graphs}\label{Primegraphssubclasses}
	
\end{center}

\section{2-Frobenius and solvable Frobenius USR Groups}
One of the main objectives of this article is to classify the prime graphs realizable by solvable USR groups. Frobenius and $2$-Frobenius groups play an important role in studying the prime graphs of solvable groups. As stated in \Cref{ConnectedComponents}, the prime graph of a solvable group is disconnected if and only if the group is a Frobenius group or a $2$-Frobenius group. In this section, we completely classify the prime graphs realizable by solvable Frobenius and $2$-Frobenius groups.

 It is known that every $2$-Frobenius group is solvable whereas Frobenius groups need not be solvable. Moreover, there exists a Frobenius USR group that is not solvable (see \cite[Main theorem]{PV26}). Using the classification of Frobenius USR groups provided in \cite{PV26}, we obtain that the disconnected graphs (c), (e), (g), (h), (l) and (p) in \Cref{Examplesprime}, together with $\Gamma_1 = (2 - 5~~ 3)$, are realized by solvable Frobenius uniformly semi-rational groups (see \Cref{Primegraphssubclasses}). Consequently, we have the following proposition.

\begin{proposition}\label{FrobeniusUSR}
The prime graphs realizable by non-trivial solvable Frobenius groups are precisely the graphs (c), (e), (g), (h), (l) and (p) in \Cref{Examplesprime}, together with $\Gamma_{1}$, where $\Gamma_{1} = (2 - 5 ~~ 3)$.
\end{proposition}
	
\begin{remark}
Note that $\Gamma_{1} = (2 - 5 ~~ 3)$ is not realizable by any inverse semi-rational group but by a USR group. It has been observed in \cite{BKMdR24} that if $G$ is an inverse semi-rational group and $|\pi(G)|\geq 3$, then $2-3 \in \bm{\Gamma}$. This is no longer true in USR groups, in general. Though, we shall observe in Section~$4$, that this holds good for metanilpotent USR groups.
	\end{remark}

We next consider $2$-Frobenius USR groups.
\begin{proposition} \label{2-frobenius}
The prime graphs realizable by non-trivial  $2$-Frobenius groups are precisely (c), (e), (g), (h), (l) and (p) in \Cref{Examplesprime}.
\end{proposition}
	
	\begin{proof} Let $\bm{\Gamma} = \Gamma(G)$ be the prime graph realized by a non-trivial $2$-Frobenius USR group $ G$. It follows from \Cref{Frobenius3Vertices} that $\bm{\Gamma}$ has at most $3$ vertices. In view of \Cref{ConnectedComponents}, we have that either $\bm{\Gamma}$  has exactly two vertices and no edge or $\bm{\Gamma}$ has three vertices and one edge. 
\begin{description}		
\item[Case I] \underline{$\bm{\Gamma}$ has two vertices.} \\
In this case, $\pi(G)=\{2,3\}$, $\{2,5\}$ or $\{3,7\}$. Note that (c) $= (2~~3)$, (e) $= (2~~5)$ and (g) $= (3~~7)$ are respectively realized by $2$-Frobenius inverse semi-rational groups $C_2^2 \rtimes (C_3 \rtimes_{Fr} C_2)$, $C_2^4 \rtimes (C_5 \rtimes_{Fr} C_4)$ and $C_3^6 \rtimes (C_7 \rtimes_{Fr}C_3)$, as observed in \cite[Proposition~5.3]{BKMdR24}. \\
		  
\item[Case II] \underline{$\bm{\Gamma}$ has three vertices.} 
\begin{itemize}
\item[Subcase I]\underline{$\pi(G)=\{2,3,7\}.$}\\
 If $\pi(G)=\{2,3,7\}$, then $\bm{\Gamma} $ has exactly one of the edges $2-3$, $3-7$ or $2-7$. By \Cref{Edges}, neither $3-7$ nor $2-7$ can be the only edge of $\bm{\Gamma}$. Furthermore, \linebreak $(l) = (2-3~~ 7)$ is realized by the $2$-Frobenius inverse semi-rational group \linebreak $C_2^6 \rtimes (C_7 \rtimes_{Fr} C_6)$, as observed in \cite[Proposition~5.3]{BKMdR24}. 
\item[Subcase II] \underline{$\pi(G)=\{2,3,13\}.$}\\
If $\pi(G)=\{2,3,13\}$, then the only possibility for $\bm{\Gamma}$ is (p) $=(2 - 3~~ 13)$ and we prove that there exists a $2$-Frobenius USR group $G$ such that $\Gamma(G)=(2-3~~13)$. For this, consider the matrices

\[
\begin{adjustbox}{max width=\textwidth}
	\(
	\small
	\setlength{\arraycolsep}{3.2pt}
	\begin{array}{@{}c@{\qquad}c@{}}
		A =
		\left(
		\begin{array}{rrrrrrrrrrrr}
			0 & 0 & 0 & 0 & 0 & 0 & 0 & 0 & 0 & 0 & 0 & -1 \\
			1 & 0 & 0 & 0 & 0 & 0 & 0 & 0 & 0 & 0 & 0 & -1 \\
			0 & 1 & 0 & 0 & 0 & 0 & 0 & 0 & 0 & 0 & 0 & -1 \\
			0 & 0 & 1 & 0 & 0 & 0 & 0 & 0 & 0 & 0 & 0 & -1 \\
			0 & 0 & 0 & 1 & 0 & 0 & 0 & 0 & 0 & 0 & 0 & -1 \\
			0 & 0 & 0 & 0 & 1 & 0 & 0 & 0 & 0 & 0 & 0 & -1 \\
			0 & 0 & 0 & 0 & 0 & 1 & 0 & 0 & 0 & 0 & 0 & -1 \\
			0 & 0 & 0 & 0 & 0 & 0 & 1 & 0 & 0 & 0 & 0 & -1 \\
			0 & 0 & 0 & 0 & 0 & 0 & 0 & 1 & 0 & 0 & 0 & -1 \\
			0 & 0 & 0 & 0 & 0 & 0 & 0 & 0 & 1 & 0 & 0 & -1 \\
			0 & 0 & 0 & 0 & 0 & 0 & 0 & 0 & 0 & 1 & 0 & -1 \\
			0 & 0 & 0 & 0 & 0 & 0 & 0 & 0 & 0 & 0 & 1 & -1
		\end{array}
		\right),
		&
		B =
		\left(
		\begin{array}{rrrrrrrrrrrr}
			0 & 0 & 0 & 0 & 0 & 0 & 0 & 0 & 0 & 1 & 0 & 0 \\
			0 & 0 & 0 & 0 & 0 & 0 & 1 & 0 & 0 & 0 & 0 & 0 \\
			0 & 0 & 0 & 1 & 0 & 0 & 0 & 0 & 0 & 0 & 0 & 0 \\
			1 & 0 & 0 & 0 & 0 & 0 & 0 & 0 & 0 & 0 & 0 & 0 \\ 
			0 & 0 & 0 & 0 & 0 & 0 & 0 & 0 & 0 & 0 & 1 & 0 \\
			0 & 0 & 0 & 0 & 0 & 0 & 0 & 1 & 0 & 0 & 0 & 0 \\
			0 & 0 & 0 & 0 & 1 & 0 & 0 & 0 & 0 & 0 & 0 & 0 \\
			0 & 1 & 0 & 0 & 0 & 0 & 0 & 0 & 0 & 0 & 0 & 0 \\
			0 & 0 & 0 & 0 & 0 & 0 & 0 & 0 & 0 & 0 & 0 & 1 \\
			0 & 0 & 0 & 0 & 0 & 0 & 0 & 0 & 1 & 0 & 0 & 0 \\
			0 & 0 & 0 & 0 & 0 & 1 & 0 & 0 & 0 & 0 & 0 & 0 \\
			0 & 0 & 1 & 0 & 0 & 0 & 0 & 0 & 0 & 0 & 0 & 0
		\end{array}
		\right)
	\end{array}
	\)
\end{adjustbox}
\]

	Reducing these matrices modulo $2$ gives a $12$-dimensional faithful representation of $H = C_{13} \rtimes_\Frob C_6$ over $\mathbb{F}_2$. So, we form the semi-direct product $G = C_2^{12} \rtimes (C_{13} \rtimes_\Frob C_6)$, where $C_2^{12}$ is identified as a $12$-dimensional $\F_2$-vector space and $C_{13}\rtimes_\Frob C_6$ acts by the above representation. Observe that $A$ is the companion matrix of the $13^{\mathrm {th}}$ cyclotomic polynomial. Hence, $F_2(G) = C_2^{12} \rtimes_\Frob C_{13}$ is a Frobenius group and $G$ is a $2$-Frobenius group with the desired prime graph $(2-3   13)$.  
						
In order to see that $G$ is a USR group, we begin by observing that order of every element of $G$ is either $13$ or divides $12$.
Let $a$ and $b$ be generators of $H$ of order $13$ and $6$, respectively. It is clear that $a$ is a semi-rational element in $H$ with $m_a=11$. So, in order to prove that $G$ is USR, it suffices to check that each element $g$ of order $12$ in $G$ is semi-rational  with $m_g=11$. Let $g$ be an element of order $12$ in $G$. Consider the Hall $13'$-subgroup $K = N\rtimes C_6$, where $N = F(G) \simeq C_2^{12}$.
By Hall's Theorem, $g$ is conjugate to an element in $K$ and therefore we may assume that $g\in K$.
The centre $\mathcal{Z}(K)$ of $K$ is generated by  $f_1$ and $f_2$, where $ f_1 = (1,0,1,1,0,0,0,0,1,1,0,1) $ and $f_2 = (0,1,0,0,1,1,1,1,0,0,1,0)$. Let $Y_1$ be the subset of $N$ consisting of all elements for which the number of zero entries at the coordinate positions $1$, $3$, $4$, $9$, $10$ and $12$ is odd. Similarly, let $Y_2$ be the subset of $N$ consisting of all elements for which the number of zero entries at the coordinate positions $2$, $5$, $6$, $7$, $8$ and $11$ is odd. Let $X=Y_1\cup Y_2$, so that $|X|=3 \cdot 2^{10}$ and  $X=X_1 \dot{\cup} X_2 \dot{\cup} X_3,$ where $X_1=Y_1\cap Y_2, X_2=Y_1{\setminus} Y_2 ~\mathrm{and}~ X_3=Y_2{\setminus} Y_1$.
						
Then $g=xb$ or $g=xb^{-1}$ with $x\in X$ and we may assume without loss of generality that $g=xb$. We now compute the centralizer $C_K(g)$.
Indeed, let $h\in \C_K(g)$ and write $h=yb^i$ with $y\in C_2^{12}$ and $0\leq i \leq 5$. Since $\C_K(g)\cap N=\mathcal{Z}(K)$ and $g^ih^{-1}\in \C_K(g)\cap N$, we have that $ h = g^{i}, f_1 g^{i}, f_2 g^{i}$ or $f_1 f_2 g^{i}$.
Now,
\[ g^6 = x x^b x^{b^2} x^{b^3} x^{b^4} x^{b^5}=
\begin{cases}
f_1 f_2, & \text{if } x \in X_1, \\
f_1,     & \text{if } x \in X_2,\\
f_2,    & \text{if } x \in X_3.
\end{cases}
\]
Therefore,
\[h =
\begin{cases}
g^{i} ~\mathrm{or}~ \; f_1 g^{i}  ~\mathrm{or}~  \; f_1 g^{6+i}  ~\mathrm{or}~  \; g^{6+i}, & \text{if } x \in X_1\cup X_3, \\[4pt]
g^{i} ~\mathrm{or}~  \; g^{6+i} ~\mathrm{or}~  \; f_2 g^{i} ~\mathrm{or}~  \; f_2 g^{6+i}, & \text{if } x \in X_2.
\end{cases}
\]
It follows that $|C_K(g)| = 24$. Hence, the conjugacy class of $g$ in $K$ has cardinality $|K|/24 = 2^{10}$. For $j=1,2,3$, the set $X_j b = \{\, x' b : x' \in X_j \,\}$ is closed under conjugation in $K$. Hence $X_j b$ contains the conjugacy class of $g$ in $K$, if $g\in X_jb$. Also, $|X_j b| = 2^{10}$. It follows that for $g=xb$,  the conjugacy class of $g$ in $K$ is $X_j b $, if $x\in X_j$.
Now,
\[g^7 =
\begin{cases}
f_1 f_2 g, & \text{if } x \in X_1, \\
f_1 g,     & \text{if } x \in X_2, \\
f_2 g,     & \text{if } x \in X_3.
\end{cases}
\]
Equivalently,
\[g^7 =
\begin{cases}
f_1 f_2 x b \in X_1 b, & \text{if } x \in X_1, \\
f_1 x b \in X_2 b,     & \text{if } x \in X_2, \\
f_2 x b \in X_3 b,     & \text{if } x \in X_3.
\end{cases}
\]
Therefore, $g$ and $g^7$ are conjugate in $K$ and hence $g \text{ is semi-rational, with } m_g = 11$.

\item[Subcase III]\underline{$\pi(G)=\{2,3,5\}.$}\\ If $\pi(G)=\{2,3,5\}$, then $\bm{\Gamma} $ has exactly one of the edges $2-3$, $3-5$ or $2-5$. By \Cref{Edges}, $3-5$ cannot be the only edge of $\bm{\Gamma}$. We prove that $\bm{\Gamma}= (2 -3~~ 5)$ is realized by a $2$-Frobenius USR group whereas $\bm{\Gamma}= (3~~2-5)$ is not realizable by a $2$-Frobenius USR group.

As in Subcase II of Case II, one can define a $2$-Frobenius USR group $C_3^4 \rtimes (C_5 \rtimes_{Fr} C_2)$ which realizes the graph (h) using 
\[
\left\langle A=
\begin{pmatrix}
0 & 0 & 0 & -1 \\
1 & 0 & 0 & -1 \\
0 & 1 & 0 & -1 \\
0 & 0 & 1 & -1
\end{pmatrix},
\quad
B= \begin{pmatrix}
0 & 0 & 0 & 1 \\
0 & 0 & 1 & 0 \\
0 & 1 & 0 & 0 \\
1 & 0 & 0 & 0
\end{pmatrix}
\right\rangle.
\]
as the matrix realization of the upper Frobenius group.
		
It only remains to prove that there is no $2$-Frobenious USR group realizing the prime graph $(2-5 ~~3)$. Assume that $G$ is a $2$-Frobenius group realizing this graph. Then using \cite[Main theorem]{PV26}, $G/F(G) \simeq C_5 \rtimes_{Fr} C_4,\ C_5 \rtimes_{Fr} C_2$ or $C_3 \rtimes_{Fr} C_2$. If $G/F(G) \simeq C_5 \rtimes_{Fr} C_4, \ \text{or}\ C_5 \rtimes_{Fr} C_2$, then $3 \mid |F(G)|$. Hence, neither $2$ nor $5$ divides $|F(G)|$. Thus $F(G)$ is a $3$-group, which is not possible since there exists $g$ such that $|g| = 2 \cdot 5 = 10$.
Therefore, $G/F(G) \simeq C_3 \rtimes_{Fr} C_2$ and hence $F(G)$ is either a $5$-group or a $\{2,5\}$-group. Suppose first that $F(G)$ is a $5$-group, so that Sylow $2$-subgroup of $G$ is $C_2$. Consider an element $g \in G$ such that $|g| = 10$, so that $B_G(g)\simeq C_2$. Since $B_G(g)=N_G(\langle g\rangle)/C_G(g)$
and $2\mid |C_G(g)|$, it follows that $4\mid |N_G(\langle g\rangle)|$, and hence $4\mid |G|$, a contradiction.
 If $F(G)$ is a $\{2,5\}$-group, then $G/\O_2(G)$ is uniformly semi-rational group having the same prime graph. Then $G/\O_2(G)$ is a Frobenius group or a
  $2$-Frobenius group. But by \cite[Main theorem]{PV26}, $G/\O_2(G)$ can not be a Frobenius group and hence $G/\O_2(G)$ is a $2$-Frobenius group which is not possible, by using similar arguments as above.
\end{itemize}
\end{description}
\end{proof}

\section{Metanilpotent groups} In this section, we consider metanilpotent USR groups. A group $G$ is called metanilpotent if it has a normal subgroup $N$ such that both $N$ and $G/N$ are nilpotent.
In \Cref{Examplesprime}, we provide a list of the possible prime graphs of metanilpotent USR groups.
It may be noted that the groups in \Cref{Examplesprime} are constructed using inverse semi-rational groups in \cite{BKMdR24}, Frobenius USR groups in \cite{PV26} and the fact that if $G$ is a USR group and $H$ is a rational group then $G \times H$ is a USR group (see \cite[Proposition~4.9]{CM25}). We now eliminate prime graphs which cannot be realized by a metanilpotent USR group.\\

If $G$ is a metanilpotent USR group, then $\pi(G)$ is one of the choices listed in \Cref{PrimeSpectrum}. If $|\pi(G)| \leq 2$, then each of the prime graphs (a) to (g) as listed in \Cref{PrimeSpectrum} is realized by a metanilpotent USR group, as the indicated groups are all metanilpotent USR groups. Also, the connected graph with $\pi(G) = \{3, 7\}$ is not possible in view of \Cref{Edges}. So, we next consider the cases when $|\pi(G)| \geq 3$.
\begin{description}
\item[Case I] $\underline{|\pi(G)| =3}.$\\ Firstly, the prime graphs with three vertices and no edges are excluded by \Cref{ConnectedComponents}. As listed in \Cref{Examplesprime}, the graphs (h)-(r) are realized by metanilpotent USR groups. Amongst the remaining possibilities, we only have $(2-5~~3)$ and $(2-3-13)$ as the possible prime graphs in view of \Cref{Edges}. We check that $(2-5~~3)$ is not possible, though  $(2-3-13)$ remains undecided. Consider $(2-5~~3)$, which is disconnected and hence is essentially a prime graph of a Frobenius or a $2$-Frobenius group (see \Cref{ConnectedComponents}). Note that by \cite[Main theorem]{PV26}, there does not exist a Frobenious metanilpotent USR group realizing this graph. Also, no $2$-Frobenius group is a metanilpotent group. Thus $(2-5~~3)$ is not realizable by a metanilpotent USR group. \\

\item[Case II] $\underline{|\pi(G)| = 4}.$\\ Since there are many possibilities in this case, we further divide this case into subcases as per elements of $\pi(G)$. Similar to Case I, we decide the realizability of all prime graphs, except for the complete prime graph with $\pi(G) = \{2, 3, 7, 13\}.$\\

\begin{itemize}

\item[Subcase I] $\underline {\pi(G) = \{2, 3, 5, 7\}}.$ \\
The graphs (s)-(u) in \Cref{Examplesprime} are realized by metanilpotent USR groups, as specified in \Cref{Examplesprime}.

We observe that if $\bm{\Gamma}$ is the prime graph of a metanilpotent USR group $G$, then $\bm{\Gamma}$ has the edge $5-7$. This is because $G/F(G)$ is a nilpotent USR group and hence its prime spectrum is contained in $\{2, 3\}$. Therefore $G_{\{5, 7\}}$ is contained in $F(G)$. Thus $F(G)$ and hence $G$ has elements of order $5 \cdot 7$. Using this fact and  \Cref{Edges}, along with realizability of graphs (s)-(u) by metanilpotent USR groups, we only need to check the possibility of prime graph with $5$ edges in which $2-5$ edge is missing. If $G$ is a metanilpotent USR group of minimal order with such a graph, then \Cref{CyclicSylow} implies that $G_2 \simeq C_2,\, C_4,\, Q_8 \text{ or } Q_{16}$. Since $G$ has an element of order $2 \cdot 5$, then $G_2$ is not isomorphic to $C_2$ and using  \Cref{G2G3Exclusions}, $G_2$ is not isomorphic to $C_4$. Also, $G_2 \not\simeq Q_{16}$ because otherwise $G/F(G) \simeq Q_{16} \times K$, where $K$ is a $3$-group. This contradicts the fact that $G/F(G)$ is a USR group. It only remains to prove that $G_2 \not\simeq Q_8$. Assume $G_2 \simeq Q_8$. Observe first that $3 \nmid |F(G)|$, because otherwise there exists an element of order $3 \cdot 5 \cdot 7$, which is not possible. Hence, $F(G)$ is a Hall $\{5,7\}$-subgroup, so that $F(G)\simeq G_5 \times G_7.$ We claim that $G_5$ is an elementary abelian group whose every element is rational in $G$. To check that $G_5$ is elementary abelian minimal normal subgroup of $G$, consider a  minimal normal subgroup $A$ properly contained in $G_5$ which yields that $G/A$ is a metanilpotent USR group realizing the same graph as $G$, a contradiction. Thus, $G_5$ is an elementary abelian $5$-group. Further, we check that for any element $x \in G_5$, $x$ is rational in $G$. For this, consider an element $y$ of order $7$ in $G_7$. Let $g = xy$, so that $|g| = 35$ and $B_G(g)$ contains an element $\phi$ of order $12$, where $\phi : \langle g \rangle \to \langle g \rangle$ is such that $\phi(g) = g^i$ with $\gcd(i,35) = 1$. Now, $\phi(g^7) = g^{7i} = l^{	-1} g^7 l$ for some $l \in G$, whose order is divisible by $12$. Clearly, elements of the set $\{g^7, g^{7i}, g^{7i^2}, g^{7i^3}\}$ belong to the same conjugacy class. To check that, these are distinct, suppose $g^{7i^{k_1}}
= g^{7i^{k_2}}$ where $k_1 \neq k_2$ and $0\leq k_1, k_2 \leq 3$. Then $l^{-k_1} g^7 l^{k_1} = l^{-k_2} g^7 l^{k_2} \Rightarrow l^{k_2-k_1} g^7 l^{-(k_2-k_1)} = g^7$ and thus there exists a $2$-element which commutes with a $5$-element, a contradiction. Hence, $x$ is rational in $G$ and every element of $G_5$ is rational in $G$.

Take
	\[
	H = \frac{G}{G_7} \simeq G_5 \rtimes (Q_8 \times G_3).
	\]
 Every element of order $5$ in $H$ is rational in $H$. $G_5 Q_8$ is the unique Hall $\{2,5\}$-subgroup of $H$. Considering, $G_5$ as an $F_5Q_8$ module and proceeding as in Case I of \cite[Proposition~4.4]{BKMdR24}, we obtain that $G_5 \simeq C_5 \times C_5$, and consequently $$G \simeq (C_5^2 \times G_7) \rtimes (Q_8 \times G_3).$$
	
The group $G$ acts on $C_5^2$ by conjugation. Hence, we obtain a homomorphism \linebreak $\phi : G \to \mathrm{GL}_2(5)$, yielding that 
$|G/ \operatorname{ker} \phi | \mid 24.$ If $3$ does not divide $|G/ \operatorname{ker} \phi |$, i.e., every $3$-element lies in $\operatorname{ker} \phi$, and therefore commutes with every element of order $5$, then $G$ contains an element of order $3 \cdot 5 \cdot 7$, which contradicts the assumption of this case that  $G_2 \simeq Q_8$. Therefore, $| G/ \operatorname{ker} \phi \mid = 24.$ Since $G/ \operatorname{ker} \phi$ is a subgroup of $\mathrm{GL(2,5)}$, it follows that $G/ \operatorname{ker} \phi \simeq C_3 \rtimes C_8, \ C_{24}, \ \mathrm{SL}_2(3)$ or $C_4 \times S_3$ which is not possible, as $G/ \operatorname{ker} \phi$ is a nilpotent USR group.
\item[Subcase II] $\underline{\pi(G) = \{2,3,5,13\}}.$\\ If $\bm{\Gamma}$ is the prime graph of a metanilpotent USR group $G$, then proceeding as in Subcase I, we obtain that $\bm{\Gamma}$ has the edge $5-13$. Using \Cref{Edges}, we obtain that there are possibly two prime graphs realizable by a metanilpotent USR group with $\pi(G) = \{2,3,5,13\}$. Indeed, both these graphs are realized as observed in (v) and (w) of \Cref{Examplesprime}. 

\item[Subase III] $\underline {\pi(G) = \{2,3,7,13\}}.$\\  If $\bm{\Gamma}$ is the prime graph of a metanilpotent USR group $G$, then as observed earlier $\bm{\Gamma}$ has the edge $7-13$. In view of \Cref{Edges}, the only possible prime graphs with this prime spectrum are the complete prime graph and the graph with $5$ edges in which $2-13$ edge is missing. We prove that the prime graph with $5$ edges is not possible whereas the complete graph remains undecided.

\end{itemize}

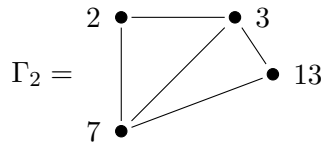
\begin{center}

\begin{tikzpicture}	
\node[anchor=east] at (-0.5,0.75) {$\Gamma_2 = $};
	
	\node[align=center, font=\scriptsize, text width=4cm] at (0.75,-0.5) {};				
	\node[label=west:{$2$}] at (0,1.5) (2){};
	\node[label=east:{$3$}] at (1.5, 1.5) (3){};
	\node[label=west:{$7$}] at (0, 0) (7){};
	\node[label=east:{$13$}] at (2, 0.75) (13){};
	
	\foreach \p in {2,3,7,13}{
		\draw[fill=black] (\p) circle (0.075cm);
	}
	
	\draw (2)--(3);
	\draw (3)--(13);
	\draw (3)--(7);
	\draw (7)--(13);
	\draw (2)--(7);

\end{tikzpicture}

\end{center}

Let $G$ be a finite metanilpotent USR group of minimal order such that $\Gamma_2 = \Gamma(G)$. Proceeding similar to Subcase I, we get that $G_2 \simeq Q_8$, $F(G)$ is a Hall $\{7, 13\}$-subgroup of G and $G_{13}$ is an elementary abelian subgroup of $G$. Let $x \in G_{13}$ and let $y \in G_7$ be an element of order $7$. Suppose $g = xy$, so that $|g| = 13 \cdot 7$ and $B_G(g)$ contains an element of order $12$. Thus, there exists an automorphism $\phi : \langle g \rangle \to \langle g \rangle$ of order $12$ such that $\ \phi(g) = g^i,$ where $\gcd(i,35)=1$. Now, $\phi(g^7) = g^{7i} = l^{-1} g^7 l$, where $l \in G$ and $12$ divides $|l|$. We claim that $g^7$ is rational in $G$. Suppose that $g^7$ is not rational in $G$. Since $12$ divides the order of $l$, the elements of $S_1 = \{g^7, g^{7i}, g^{7i^2}, g^{7i^3}\}$ lie in a single conjugacy class. If $g^{7i^4} \notin S_1$, then we have a contradiction as $g^7$ is not a rational element. If $g^{7i^4} \in S_1$, then there exist elements $g^{7r_1}, g^{7r_2} \in \langle g^7 \rangle$ such that
 $S_m \cap S_n = \phi$ for all distinct $m, n \in \{1, 2, 3\}$, and $S_1 \cup S_2 \cup S_3 = \langle g^7 \rangle$, where
  $S_2 = \{g^{7r_1}, g^{7ir_1}, g^{7i^2r_1}, g^{7i^3r_1}\}$ and $S_3 =  \{g^{7r_2}, g^{7ir_2}, g^{7i^2r_2}, g^{7i^3r_2}\}$. Moreover, for each $t \in \{1, 2, 3 \}$, all elements of $S_t$ lie in a single conjugacy class, and the elements of $S_m$ are not conjugate to the elements of $S_n$ whenever $m \neq n$. Again, using semi-rationality of $g^7$, we get a contradiction. Thus every element of order $13$ is rational in $G$. Take $K$ to be Hall $\{2,3,7\}$-subgroup of $G$. Then, the group homomorphism $\psi : K \rightarrow \Aut(G_{13})$ turns $G_{13}$ into $\mathbb{F}_{13}K$-module, which is irreducible, as $G_{13}$ is a minimal normal subgroup of $G$. Let $\chi$ be the Brauer character of $K$ afforded by $G_{13}$. Then $\chi$ is actually an ordinary character as $13 \nmid |K|$. By \cite[Lemma~3]{Ten12} and the remark following it, the field of character values of $\chi$ is contained in field of character values of each of its irreducible constituents. But then $[\mathbb{Q}(\chi) : \mathbb{Q}] \leq 2.$  Since, all the $13$-elements of $G$ are rational in $G$, the action of $K$ on $G_{13}$ has the $12$-eigen value property. If $L = \mathbb{Q}(\chi)$ then $k = [L(\zeta_{13}) : L] = 6$ or $12$, where $\zeta_{13}$ is a primitive $13^{th}$ root of unity. Also, $[\mathbb{Q}(\chi) : \mathbb{Q}] \leq 2$, which is a contradiction in view of \cite[Theorem~B(c)]{Soa84}.
  
\item[Case III] $\underline{|\pi(G)|=5}.$\\Here, $\pi(G)=\{2,3, 5,7,13\}.$ If $\bm{\Gamma}$ is the prime graph of a metanilpotent USR group $G$, then $5 - 7$, $7 - 13$ and $5 - 13 \in \bm{\Gamma}$ and using \Cref{Edges}, we get that $2 - 7$, $3 - 5$, $2 - 3$, $2 - 5$, $2 - 13$, $3 - 7$ and $3 - 13 \in \bm{\Gamma}$ i.e., $\bm{\Gamma}  $ is a complete graph and this graph $\bm{\Gamma} = (z^*)$ remains undecided.
\end{description}
Therefore, we have the following theorem:
\begin{theorem}\label{Metanilpotent}
The prime graphs realizable by non-trivial metanilpotent groups are precisely \linebreak (a) - (w), (x*), (y*) and (z*) in \Cref{Examplesprime}.
\end{theorem}
\begin{remark}\label{Remain}
The graphs which remain undecided in Case II and Case III are realized or not realized by metanilpotent USR groups simultaneously. Because if $G$ is a metanilpotent USR group whose prime graph is $(y^*)$, then $G \times (C_5^2 \rtimes _{Fr} Q_8)$ will be a metanilpotent USR group whose prime graph is $(z^*)$. Conversely, if $G$ is a metanilpotent USR group whose prime graph is $(z^*)$, then $G/O_5(G)$ will be a metanilpotent USR group  whose prime graph is $(y^*)$. 
\end{remark}

\subsection{Nilpotent-by-abelian and metabelian USR groups}
\begin{proposition}\label{nilpotent-by-abelian}
	The following assertions are equivalent for a prime graph $\bm{\Gamma}$:
	\begin{itemize}
		\item[(i)] $\bm{\Gamma}$ is realizable by a nilpotent-by-abelian USR group. 
		\item[(ii)] $\bm{\Gamma}$ is realizable by a metabelian group.
		\item[(iii)] $\bm{\Gamma}$ is trivial or one of the graphs in (a)-(r) in \Cref{Examplesprime}.
	\end{itemize}
\end{proposition}

\begin{proof}
The implication that (ii) implies (i) is clear and (iii) implies (ii) follows from \Cref{Primegraphssubclasses}. 

We thus show that (i) implies (iii). The graphs (a)-(r) in \Cref{Examplesprime} are realized by nilpotent-by-abelian USR groups as specified in \Cref{Primegraphssubclasses}. We further check that no graph other than (a)-(r) is realizable by a non-trivial nilpotent-by-abelian USR group. If $|\pi(G)| \leq 3$ and $\bm{\Gamma}$ is not one of (a)-(r), then $\bm{\Gamma} = (x^*)$. We prove that no nilpotent-by-abelian USR group realizes this graph. For, if $G$ is such a group then $G_{13} \leqslant F(G)$ i.e., $G_{13}$ is a normal subgroup of $G$. But, there does not exists an element of order $26$. Thus, using \Cref{CyclicSylow},
we get $G_2 \simeq C_2, \ C_4, \ Q_8,$ or $ Q_{16}$. As observed in Subcase~I of Case~II, $G_2 \not\simeq C_2$ and $G_2 \not\simeq C_4$. Moreover, $G_2$ is isomorphic to neither $Q_8$ nor $Q_{16}$, since $G/F(G)$ is abelian, as $G$ is nilpotent-by-abelian.

Now, if $|\pi(G)| \geq 4$ and $\pi(G) = \{2, 3, 5, 7, 13\}$ or $\{2, 3, 5, 13\}$, then there exists $g \in \mathcal{Z}(F(G))$ such that $|g| = 5 \cdot 13$ and hence, $B_G(g)$ contains an element of order $12$, which is a contradiction because exponent of $G/F(G)$ divides $4$ or $6$, as $G/F(G)$ is abelian.

Further, if $\pi(G) = \{2, 3, 5, 7\}$ or $\{2, 3, 7, 13\}$, then none of the graphs (s), (t), (u), (v) or (w) is realizable by a nilpotent-by-abelian USR group. The non-realizability of graph (s) is proved by arguments analogous to those used to prove the non-realizability of graph (x*). We prove the result for graph (t), and the remaining cases follow similarly. Let $G$ be a nilpotent-by-abelian USR group of minimal order which realizes the graph (t). Then $G \simeq (C_5^m \times C_7^n) \rtimes (C_2^l \times C_3)$, where $m, n, l \in \mathbb{N}$. Since $4$ does not divide $\operatorname{exp(G)}$, every $5$-element of $G$ is not rational but semi-rational. Since $G$ contains an element $g$ of order $15$, $B_G(g)$ contains an element $\phi$ of order $2$ such that $\phi(g) = g^i = y^{-1}gy$, where $y \in G$, $\gcd(i, 35) = 1$ and $\phi$ is identity on $5$-part of $g$. Then $g^5$ and $g^{5i}$ are distinct and conjugate elements. Thus $g^{5i} = g^{-5}$ and $y^{-1}g^5y = g^{-5}$, where $2$ divides order of $y$. Hence, $(yK)^{-1}(g^5K)(yK) = g^{-5}K$, where $K$ is the Hall $\{5, 7\}$-subgroup of $G$ which is a contradiction, thereby completing the proof.
\end{proof}

\subsection{Abelian-by-cyclic USR groups}
\begin{proposition}\label{abelian-by-cyclic}
The prime graphs realizable by non-trivial abelian-by-cyclic groups are precisely (a)-(q) in \Cref{Examplesprime}.
\end{proposition}
\begin{proof}
The graphs (a)-(q) in \Cref{Examplesprime} are realized by abelian-by-cyclic groups as specified in \Cref{Primegraphssubclasses}(V). In view of \Cref{nilpotent-by-abelian}, it suffices to show that there is no abelian-by-cyclic USR group realilzing the graph (r). Let $G$ be an abelian-by-cyclic group of minimal order realizing (r). Then there exists $N \unlhd G$ such that $N$ is abelian and $G/N$ is cyclic. In particular, $N \leqslant F(G)$, and hence

\[G/F(G) \simeq C_2,\ C_3,\ C_4\ \text{or } C_6.\]

Since both the groups $G/O_2(G)$ and $G/O_3(G)$ are USR, the subgroup $F(G)$ cannot contain a Sylow $2$-subgroup or a Sylow $3$-subgroup. It follows that $G/F(G) \simeq C_6$. Consequently, $3 \nmid |F(G)|$, for otherwise there exists $g \in \mathcal{Z}(F(G))$ of order $3 \cdot 13$, a contradiction. Thus $G_3 \simeq C_3$, which is not possible by \Cref{G2G3Exclusions}.
\end{proof}

\subsection{Cyclic-by-abelian USR groups}
\begin{proposition}\label{cyclic-by-abelian}
The prime graphs realizable by non-trivial cyclic-by-abelian groups are precisely (a)-(g) and  (j)-(r) in \Cref{Examplesprime}.
\end{proposition}

\begin{proof}
The graphs (a)-(g), (j)-(r) in \Cref{Examplesprime} are realized by cyclic-by-abelian groups as specified in \Cref{Primegraphssubclasses} (VI). In view of \Cref{nilpotent-by-abelian}, it suffices to show that no cyclic-by-abelian USR group realizes the graphs (h) and (i).

First consider the graph (h) = $(2 - 3 \ \ 5)$, which is disconnected. Hence, it must be realized by a Frobenius or a $2$-Frobenius group (see \Cref{ConnectedComponents}). But, by \cite[Main Theorem]{PV26}, there is no Frobenius cyclic-by-abelian USR group realizing this graph. Moreover, no $2$-Frobenius group is cyclic-by-abelian. Therefore, the graph $(2 - 3 \ \ 5)$ is not realizable by any metanilpotent USR group.

Next, consider the graph (i) = $(5 - 2 - 3)$, and suppose that $G$ is a cyclic-by-abelian group of minimal order realizing this graph. Then, there exists a normal subgroup $N$ of $G$ such that $N$ is cyclic and $G/N$ is abelian. Hence, $G/N$ has exponent dividing $4$ or $6$. Also, Sylow $5$-subgroup $G_5$ is elementary abelian and normal in $G$. Consequently, $G_5 \simeq C_5$ and using \Cref{CyclicSylow}, we get $G_3 \simeq C_3$. Since $G_{\{2,3\}}$ acts on $G_5$ by conjugation, this would force the existence of an element of order $3 \cdot 5$, which is impossible. This contradiction shows that no such group $G$ exists.
	
\end{proof}
\subsection{Metacyclic USR groups}
Clearly, the following corollary implies from the   \Cref{abelian-by-cyclic} and \Cref{cyclic-by-abelian} .
\begin{corollary} \label{metacyclic}
The prime graphs realizable by non-trivial metacyclic groups are precisely (a)-(g) and  (j)-(q) in \Cref{Examplesprime}.
\end{corollary}

\section{USR Groups generalizing Magnus property}
In $1930$, W. Magnus \cite{Mag30} introduced groups with certain property, which later came to be known as Magnus property. A group $G$ is said to have Magnus Property if for every $x,y \in G$,

\begin{equation}
\langle x^G \rangle = \langle y^G \rangle \implies x\sim y~~~\mathrm{or} ~~~ x\sim y^{-1} , \tag{MP}
\end{equation}
where $\langle x^G \rangle$ is the smallest normal subgroup of $G$ containing $x$.  Clearly, a finite group $G$ with Magnus property is an inverse semi-rational group and such groups have been studied in \cite{GM24}. If there exists $m \in \mathbb{Z}$ such that $\gcd(m,\operatorname{exp}(G)) = 1$ and for every $x, y \in G$,  
\begin{equation}
	\langle x^G \rangle = \langle y^G \rangle \implies x\sim y~~~\mathrm{or} ~~~ x\sim y^{m}, \tag{MP*}
\end{equation}
then we call such a group $G$ as an MP* group. Certainly, MP* groups extend the notion of Magnus property and are USR groups. However, there exist USR groups, for instance $S_4$, which are not MP*, but USR groups. In this section, we study the prime graphs of MP* groups which takes us a step closer to the study of prime graphs of USR groups. We first gather some basic results on MP* groups.

	\begin{proposition}
		If $G$ is a finite MP* group and $N \unlhd G$, then $G/N$ is an MP* group.
	\end{proposition}
	
	\begin{proof}
		This can be done in similar way as \cite[Corollary~2.5]{KMP23}.
	\end{proof}
	
	\begin{proposition}
		Every finite MP* group is solvable.
	\end{proposition}
	\begin{proof}
		This can be done in similar way as \cite[Proposition~1.1]{GM24}.
	\end{proof}
\begin{corollary}\label{primespectrum}
If $G$ is a finite MP* group, then $\pi(G) \subseteq \{2,3,5,7,13\}$.
	\end{corollary}

Given a group $G$ and a subgroup $H$ of $G$, recall that the normal core of $H$ in $G$, denoted by $H_G$, is the largest normal subgroup of $G$ contained in $H$, i.e., $H_G = \cap_{g \in G} ~g^{-1} H g$. A finite group $G$ is called primitive, if it admits a maximal subgroup $M$ with trivial normal core, i.e., $M_G = 1$.  Let $\Phi(G)$ denote the Frattini subgroup of $G$ which is defined to be the intersection of all maximal subgroups of $G$. We have that $G/\Phi(G)$ is a subdirect product of primitive quotients of $G$. Since MP* groups are quotient closed, $G/\Phi(G)$ is isomorphic to a subdirect product of finite primitive MP* groups. 
In the following theorem, we will classify all finite primitive MP* groups and we use this classification to obtain that all MP* groups are metanilpotent. This theorem is motivated by  \cite[Theorem~1.3]{GM24} and can be proved using similar arguments. For the convenience of reader, we provide proof of the theorem.
\begin{theorem}\label{PrimitiveGroup}
If $G$ is a finite primitive MP* group, then $G$ is one of the groups $C_2,\ C_3,\ S_3,\ A_4,\ \linebreak D_{10}, C_5 \rtimes_{Fr} C_4,\ C_7 \rtimes_{Fr} C_3,\ C_7 \rtimes_{Fr} C_6, \ M(9) = C_3^2 \rtimes_{Fr} Q_8$ and $C_{13} \rtimes_{Fr} C_6$.	\end{theorem}
\begin{proof}
 A finite primitive solvable group $G$ has the form $ V \rtimes G_0$, where $V$ is an $\mathbb{F}_p$ vector space and $G_0$ is a subgroup of $\GL(V)$ acting irreducibly on $V$  \cite[Section~2]{Sea86}. Set $ q := p^n = |V|$, the primitivity degree of $G$. Since $G_0$ acts irreducibly on $V$, the normal closure of each non-trivial element of $V$ is $V$ itself. It implies that any two non-trivial elements of $V$ are conjugate or $m$-conjugate. Hence, if $  x(\neq 0) \in V$, then $\{0\} \cup x^G \cup (mx)^G = V$, and so the group $G_0$ acts faithfully and irreducibly on $V$ with at most three orbits.
 
  If $G$ acts transitively on the non zero vectors of $V$, then $G$ is $2$-transitive. In this case, \cite[Theorem~1.1]{DMZYY23} implies that either $G_0 \leqslant \Gamma L(1,q)$, where $\Gamma L(1,q)=\mathbb{F}_q^* \rtimes \Aut(\mathbb{F}_q)$ or $q\in \{3^2,\, 3^4,\, 5^2,\, 7^2,\, 11^2,\, 23^2\} $. In the latter case, direct GAP \cite{GAP4} computations and \Cref{primespectrum} yield  that $q = 3^2$, i.e., $G = M(9)$. Thus, in any case, $G_0 \leqslant \Gamma{L}(1,q).$

  Assume that $G$ has rank $3$. Let $a$ and $b$, respectively, be the sizes of non-trivial orbits $A$ and $B$. We claim that $a = b$. For, if $a > b$, then there exist $y_1, y_2 \in A$ such that $my_1= my_2$ and thus $y_2 =l y_1$ for some integer $l$, which implies that $(l-1)my_1 = 0$. As $\gcd(m,|G|)=1$, we get that $(l-1)y_1 = 0 $ and thus $ y_2 = y_1$. Thus, $a=b$.

It follows from \cite[Theorem~1.1]{Fou69} that if $G$ has rank $3$, then either $G_0 \leqslant \Gamma{L}(1,q)$ or $G$ acts imprimitively on $V$, or it falls into a finite list of exceptions. If  $G_0 \nleqslant \Gamma{L}(1,q)$ and $G$ does not act imprimitively on $V$, then using the fact that $a=b$, we deduce that $G$ is a primitive group of degree $7^2,$ $23^2$ or $47^2$. In these cases, $G$ is contained in a maximal primitive solvable group of order $7^2 \cdot 8 \cdot 9 $, $23^2 \cdot 24 \cdot 11$ or  $47^2 \cdot 24 \cdot 46$ respectively. However, \Cref{primespectrum} and GAP \cite{GAP4} check shows that there is no MP* group of degree $7^2,$ $23^2$, $47^2$ and rank $3$.

We thus assume that $G$ acts imprimitively on $V$ and that there is a decomposition $V = V_1 \oplus V_2$ into minimal imprimitivity subspaces such that the two non zero orbits are $V_1 \cup V_2 -\{0\}$ and $V - (V_1 \cup V_2).$ The equality $a = b$ translates into $|V| = 2(|V_1| + |V_2| - 1) - 1$. Reducing modulo $p$, we deduce that $p = 3$ and one of $V_1, V_2$ has size $3$ and so does the other one (by imprimitivity). Therefore, $|V| = |V_1 \oplus V_2| = 3^2 = 9 $. Hence, $G$ has degree $9$. Again, GAP \cite{GAP4} inspection shows that no such group exists. 

We thus are left with the only possibility that $G_0 \leqslant \Gamma L(1,q)$. Let $H = G_0 \cap \mathbb{F}_q^\times$. Then both $H$ and $G_0/H$ are cyclic groups and, so $|G_0/H| \in \{1,2,3,4,6\}$. Thus $G_0$ is an extension $C_r \cdot C_s $ with $r$, $s$ positive integers and $s \leq 6$. If $t_1, t_2 \in C_r$ have order $r$, then their normal closure in $G_0$ is $C_r$. Since $G_0$ is an MP*, $t_1$ is conjugate to $t_2$ or $t_2^m$. Hence, $\varphi(r) \leq 2s \leq 12$, so $r \leq 42$. Therefore, \[
q = |V| \leq 1 + 2|G_0| \leq 1 + 2 \cdot 42 \cdot 6 = 505.
\]
Thus the primitivity degree of $G$ is at most $503$, and using \Cref{primespectrum} and GAP \cite{GAP4}, we get that  $G$ is isomorphic to one of the groups $C_2,\ C_3,\ S_3,\ A_4,\ D_{10}, \ C_5 \rtimes_{Fr} C_4,\ C_7 \rtimes_{Fr} C_3,\ C_7 \rtimes_{Fr} C_6, \ M(9) = C_3^2 \rtimes_{Fr} Q_8$ and $C_{13} \rtimes_{Fr} C_6$.

	\end{proof}

 Recall that if $G$ is a finite solvable group, then $\ell_F(G) = \ell_F(G/\Phi(G))$. Moreover, if $H \leqslant G$, then $\ell_F(H) \leq \ell_F(G)$ and $\ell_F(A \times B) = max\{\ell_F(A), \ell_F(B)\}$, for any two finite solvable groups $A$ and $B$.	Using these facts and observing that the groups listed in statement of the above theorem have Fitting length at most $2$, we obtain the following.
\begin{corollary} \label{MPmeta}
	A finite MP* group is metanilpotent.
\end{corollary}

In order to classify the prime graphs of MP* groups, it is essential to have the classification of Frobenius MP* groups. For this, we use \cite{PV26} for the classification of Frobenius USR groups. Let us recall the definition of SMP groups used in this classification. A finite group $G$ such that for every $x, y \in G$, 
\begin{equation}
	\langle x^G \rangle = \langle y^G \rangle \implies x\sim y , \tag{SMP} 
\end{equation} is called as an SMP group, short for group with strong Magnus property. Analogous to the fact that an MP group is an inverse semi-rational group, we have that an SMP group is a rational group. We now classify Frobenius MP* groups. 

 To conclude the following proposition, we repeatedly use the fact that MP* groups are quotient closed and have Fitting length at most $2$. Further, if $G = K \rtimes H$ is a Frobenius group with Frobenius kernel $K$ and Frobenius complement $H$ such that $K$ is the minimal normal subgroup of $G$, then $|K| \leq 2|H|+1$. These observations along with critical analysis of the actions in the semi-direct product yield the desired result.

\begin{proposition}\label{FrobeniusMP*}
Let $G$ be a Frobenius MP* group with Frobenius complement $H$. Then the following conclusions hold.
\begin{itemize}
\item[\textbf{(1)}] \textbf{Case: $|H|$ is even.}

\begin{itemize}
	\item[(I)] If $G$ is an SMP group, then $G$ is isomorphic to one of the following:\\
	$(i)\; C_3^n \rtimes C_2,~ n \geq 1  \qquad (ii)\; C_3^2 \rtimes Q_8.$
	\item [(II)]If $G$ is an MP group, then $G$ is isomorphic to one of the following:\\
	$(i)\; C_5^n \rtimes C_4,~ n \geq 1 \quad (ii)\; C_7^{n} \rtimes C_6, ~ n \geq 1$
	\item[(III)] If $G$ is an MP* group with $m^4 \equiv 1(mod \ \operatorname{exp(G)})$, then $G$ is isomorphic to one of the following:\\
	$(i)\; C_5^{n} \rtimes C_2,~  n \geq 1 \quad (ii)\; (C_5^a \times C_5^b) \rtimes C_4, ~ a, b \neq 0 \quad (iii)\; C_{13}^{n} \rtimes C_6, ~ n \geq 1$.
	
\end{itemize}
\item[\textbf{(2)}] \textbf{Case: $|H|$ is odd.} Then $H \cong C_3$, $G$ is an MP group, and one of the following holds:
\begin{itemize}
	\item [(I)]$K$ is an inverse semi-rational $2$-group having a fixed-point-free automorphism of order $3$.  
	In particular, $|K| = 4^a$ for some even number $a \neq 0$, and $K$ is an extension of an abelian group of exponent dividing $4$ by an abelian group of exponent dividing $4$.  
	\item[(II)]$K$ is a $7$-group having a fixed-point-free automorphism of order $3$, and every cycilc subgroup of $K$ is normalized by a conjugate of $H$. Moreover, $K$ has exponent $7$ and is an extension of an elementary abelian $7$-group by an elementary abelian $7$-group.
\end{itemize}
\end{itemize}
Conversely, all the groups listed above are MP* groups.

\end{proposition}

Consider Case (1)(III)(ii) of above proposition. The group $C_4$ has two non-isomorphic $1$-dimensional modules over $\mathbb{F}_5$ whose corresponding semi-direct products are (isomorphic) Frobenius MP* groups.  In this part, the action of $C_4$ on each direct factor of $C_5^a$ corresponds to one of these module structures, whereas the action of $C_4$ on each direct factor of $C_5^b$ corresponds to the other.

\begin{theorem} \label{MP*}
The prime graphs realizable by non-trivial MP* groups are precisely (a)-(g), (j)-(r) in \Cref{Examplesprime}.
\end{theorem}
	
\begin{proof}
The graphs (a)-(g), (j)-(r) in \Cref{Examplesprime} are realized by MP* groups as specified in \linebreak \Cref{Primegraphssubclasses}~(VIII). Now, it is enough to prove that the graphs (h), (i), (s)-(w), $(x^*)$, $(y^*)$ and $(z^*)$ are not realizable by MP* groups.

Let $\bm{\Gamma}$ be the prime graph realized by an MP* group $G$. Suppose that $\bm{\Gamma} = (h)$.
Since (h) is disconnected, by \Cref{ConnectedComponents}, either $G$ is a Frobenius group or a $2$-Frobenius group. By \Cref{FrobeniusMP*}, it is clear that there does not exist any Frobenius MP* group whose prime graph is (h). Also, no $2$-Frobenius group is an MP* group. Thus, there does not exist any MP* group whose prime graph is (h). 

Let $G$ be an MP* group of minimal order which realizes graph (i). Since $G$ is an MP* group, \Cref{MPmeta} implies that $G_5$ is a normal subgroup of $G$. Moreover, by the minimality of order of $G$, $G_5$ is a minimal normal subgroup of $G$ and $F(G) \simeq G_5$. Hence, $G_5$ is elementary abelian and thus $G_5 \simeq C_5^d$ for some positive integer $d$. Since $G_5$ is a minimal normal subgroup of $G$, normal closure of each non-trivial element of $G_5$ is $G_5$. Thus, for any non-trivial $x, y\in G_5$, we have $ {\langle x^G \rangle} =  {\langle y^G \rangle}$ and hence $y  \sim x$ or $x^m$ for some $m$, where $\gcd(m, \operatorname{exp}(G)) = 1$. Since, $G$ acts on $G_5$ by conjugation, $5^d -1 /2 $ divides $|G/C_G(G_5)|$, and $G/C_G(G_5) \leqslant \GL(d,5)$. Hence, $5^d -1$ is a $\{2, 3\}$-number. If $d$ is power of an odd prime, then $5^d -1 \not\equiv 0\pmod 3 $ and $5^d -1 \not\equiv 0\pmod 8$. Hence, neither $3$ nor $8$ divide $5^d -1$, and $5^d - 1 < 8$ is false for an odd prime power $d$. Thus $d$ is not a power of an odd prime. Similarly, if $d = 4$, we find that $5^4 - 1 = 624 = 2^4 \cdot 3 \cdot 13$, which is a contradiction as $13\nmid |G|$. Now, using the fact that if $d = ef$ where $e, f > 1$, then $5^d - 1 = (5^e - 1)\left(1 + 5^e + \cdots + 5^{e(f-1)}\right)$, the only possible values for $d$ are $1$ and $2$. If $d = 1$, then $|G/C_G(G_5)|$ divides $4$, a contradiction since $3$ divides $|G/C_G(G_5)|$. If $d=2$, then the group $ G/C_G(G_5)$ is isomorphic to a subgroup of $\mathrm{GL}(2,5)$ and $5^2 - 1 / 2= 12$ divides the order of $G/C_G(G_5)$. Since $G/C_G(G_5)$ is a $\{2,3\}$-group, it must have order $12$, $24$, $48$ or $96$. Consequently, $G/C_G(G_5)$ is isomorphic to $C_3 \rtimes C_4$, $C_{12}$, $D_{12}$, $C_3 \rtimes C_8$, $C_{24}$, $C_4 \times S_3$, $C_{24} \rtimes C_2$, $((C_4 \times C_2) \rtimes C_2) \rtimes C_3$, $\SL(2,3)$ or $\text{SL}(2,3) \rtimes C_4$ (using GAP  \cite{GAP4}), which is not possible as these groups are either not nilpotent or not MP*.

Now, we will prove that there does not exist an MP* group realizing the prime graph  $(x^*)$. Let $G$ be an MP* group of minimal order which realizes this graph. Then $G_{13} \leqslant F(G)$ and hence $G_{13}$ is a normal subgroup of $G$. Clearly, either $F(G)$ is a $13$-group or a $\{3, 13\}$-group. If $F(G)$ is a $\{3, 13\}$-group, then $G/F(G)_3$ is a Frobenius group, and by \Cref{FrobeniusMP*}, we have $G_2 \simeq C_2$, which is not true as $G$ contains an element of order $3 \cdot 13$. Thus, $F(G)$ is a $13$-group. Since, $G_{13}$ is a normal subgroup of $G$ and $G$ contains no element of order $2 \cdot 13$, \Cref{CyclicSylow} implies that $G_2 \simeq C_2,\ C_4,\ Q_8$ or $Q_{16}$. Clearly, $G_2 \not\simeq C_2,\ C_4$ or $Q_{16}$, and hence $G_2 \simeq Q_8$ possibly. Moreover, $G_{13}$ is a minimal normal subgroup of $G$, and therefore it is elementary abelian. Hence, $G_{13} \simeq C_{13}^d$ for some positive integer $d$. As above, $13^d -1 /2 $ divides $|G/C_G(G_{13})|$, and $G/C_G(G_{13}) \leqslant \GL(d, 13)$. For any odd prime $p$, $13^p - 1 \equiv 4  \pmod{8}$ and

		\[13^p - 1 \equiv
		\begin{cases}
			0 \pmod{9}, & \text{if } p = 3,\\
			3 \pmod{9}, & \text{if } p \equiv 1\pmod{3} ,\\
			6 \pmod{9}, & \text{if } p \equiv 2\pmod{3},
		\end{cases}
		\]
		
		Furthermore,
		\[
		17 \mid \frac{13^4 - 1}{2}
		\quad \text{and} \quad
		61 \mid \frac{13^3 - 1}{2}.
		\]
Arguing as above, we conclude that $d = 1$ or $2$. If $d = 1$, then $G/C_G(G_{13})$ is isomorphic to a subgroup of $\operatorname{GL}(1,13) \cong C_{12}$ which implies $2 \mid |C_G(G_{13})|$, a contradiction. If $d = 2$, then $84 \mid |G/C_G(G_{13})|$ which is also not possible.
		
Next, let $G$ be an MP* group of minimal order realizing graph (s). Since $G$ is an MP* group, $G_5$ and $G_7$ are minimal normal subgroups of $G$. Thus, using \Cref{CyclicSylow}, we get that $G_2 \simeq C_2,\ C_4,\ Q_8$ or $Q_{16}$. Clearly, $G_2 \not\simeq C_2 , \ C_4$ or $Q_{16}$, and hence $G_2 \simeq Q_8$ possibly. Since, $G_5 \simeq C_5^d$ for some positive integer $d$, the arguments used in the previous cases shows that $d = 1$ or $2$. If $d = 1$, then $|G/C_G(G_5)| \leq C_4$. Hence, $2$ divides $|C_G(G_5)|$, which is not possible. Likewise, $d = 2$ is not possible using the similar logic as above,  in the elimination of graph (i).
		
Let $G$ be an MP* group of minimal order realizing graph (t). Since $G$ is an MP* group, $G_5$ and $G_7$ are minimal normal subgroups of $G$. Thus, by \Cref{CyclicSylow}, we have $G_3 \simeq C_3$. Moreover, as $G_5$ and $G_7$ are minimal normal subgroups, they are elementary abelian. Hence, $G_5 \simeq C_5^d$ and $G_7 \simeq C_7^{d'}$ for some positive integers $d$ and $d'$. As above, we obtain $d = 1$. Since $G_7 \simeq C_7^{d'}$, it follows that $7^{d'} -1 / 2 \mid |G/C_G(G_7)|$ and $G/C_G(G_7) \leqslant GL(d', 7)$, thus $7^{d'} - 1$ is a $\{2,3\}$-number. If ${d'}$ is odd, then $7^{d'} - 1 \not \equiv 0 ~(\mathrm{mod}~4)$. Moreover, $7^4 - 1 = 2^5 \cdot 3 \cdot 5^2$ and hence ${d'} = 1$ or $2$. Clearly, $F(G) \simeq G_5 \times G_7$ and $G_2$ acts on $F(G)$ by conjugation. Therefore, $G_2 / \operatorname{ker}(\phi) \leqslant \Aut(F(G))$, where $\phi$ is a homomorphism from $G_2$ to $\Aut(F(G))$. One can observe that $\operatorname{ker}(\phi)$ is a normal subgroup of $G$, and hence trivial. Therefore, $G_2 \leqslant \Aut(F(G))$. If ${d'} = 1$, then $3 \mid |G/C_G(G_7)|$, $G/C_G(G_7) \leqslant C_6$ and $G_2 \leqslant C_4 \times C_6$. Clearly, $G_2 \not\simeq C_2$, $C_4$ or $C_4 \times C_2$. Thus $G_2$ is isomorphic to $C_2 \times C_2$. Let $G_5 \simeq \langle x \rangle$. Since $G/C_G(\langle x \rangle) \leqslant C_4$, it follows that $6 \mid |C_G(\langle x \rangle)|$. Hence, $G$ contains an element of order $2 \cdot 3 \cdot 5$, which is not possible. If ${d'} = 2$, then 24 divides $|G/C_G(G_7)|$, $G/C_G(G_7) \leqslant \GL(2,7)$ and  $G_2 \leqslant \Aut(C_5 \times C_7^2) \simeq  C_{12} \times (\SL(2,7) \rtimes C_2$. Using GAP  \cite{GAP4}, we get $G_2 \leqslant C_4 \times QD_{32} $. Consequently, the possibilities for $G_2$ are
\[
C_2,\ C_4,\ C_2 \times C_2,\ C_8,\ Q_8,\ D_8,\ Q_{16},\ D_{16},\ C_{16},\ QD_{32},
\]
or the direct product of one of these groups with $C_2$. Moreover, $|G/C_G(G_7)| \in \{24,48,96\}$ and hence the possibilities for $G/C_G(G_7)$ are
\[
C_3 \times Q_8,\ \SL(2,3),\ C_{24},\ (C_6 \times C_2) \rtimes C_2,\ C_3 \times D_8,\ 
\SL(2,3)\cdot C_2,\ C_3 \times D_{16},\ C_{48},\ C_3 \times Q_{16},\ 
C_3 \times QD_{32}.
\] 
Since $G/C_G(G_7)$ is an MP* group and a nilpotent group, it follows that $G/C_G(G_7)\simeq C_3 \times Q_8$ or $C_3 \times D_8$. Using the fact that $G_2 \times G_3$ is an MP* group, we obtain that $G_2 \simeq Q_8,\ D_8, \ Q_8 \times C_2$ or $D_8 \times C_2$,
which is a contradiction as $G / G_7$ is an MP* group and there does not exist an MP* group of this type.

Now, let $G$ be an MP* group of minimal order realizing graph (u). As above, we obtain $G \simeq (C_5 \times {C_7^d}') \rtimes (G_2 \times C_3)$, where ${d}' \in \{1, 2\}$. If ${d}' = 1$, then $G_2 \times G_3$ is a subgroup of $C_4 \times C_6$, contradicting \Cref{G2G3Exclusions}. Hence, ${d}' = 2$. In this case, $G_2 \times G_3$ is a nilpotent MP* subgroup of $\Aut(C_5 \times C_7^2)$. Using the fact that $24 \mid |G/C_G(G_7)|$ and thus $24 \mid |G|$, along with GAP \cite{GAP4} computations, we get that this is not possible.

Next, let $G$ be an MP* group of minimal order realizing graph (v). As above, we obtain $G \simeq (C_5 \times {C_{13}^d}') \rtimes (G_2 \times C_3)$, where ${d}' = 1$ or $2$. If ${d}' = 2$, then $7$ divides $|C_G(G_{13})|$, a contradiction. If ${d}' = 1$, then $G_{\{5,13\}}$ is a cyclic group. Let $g$ be a generator of  $G_{\{5,13\}}$. Then $\Aut \langle g \rangle \simeq C_{12} \times C_4$, and therefore $B_G(g) \simeq G/C_G(g)$ is an MP* group, which is a contradiction. Consequently, graph (v) is not realizable by an MP* group. Similarly, we obtain that (w) is not realizable by MP* groups.

Finally, let $G$ be an MP* group of minimal order realizing graph $(y^*)$. As above, we get that $G \simeq (C_7^d \times {C_{13}^d}') \rtimes (G_2 \times G_3)$, where $d, {d}' \in \{1, 2\}$. Since $({13^d}' - 1)/2$ divides $|G/C_G(G_{13})|$ and $7 \nmid |G/C_G(G_{13})|$, it follows that ${d}' = 1$. Suppose first that $d =1$. Let $g$ be a generator of  $G_{\{7,13\}}$. Then $36 \mid |B_G(g)|$ and hence $36 \mid |G|$. Moreover, $G_2 \times G_3 \leqslant C_6 \times C_{12}$. Thus, $G_2 \times G_3 \simeq C_6 \times C_6$ as $G_2 \times G_3$ is an MP* group. Also, $G_2 \times G_3$ acts on $G_{13}$ by conjugation, we have that $(G_2 \times G_3)/	\operatorname{ker}(\phi) \leqslant C_6$, where $\phi$ is homomorphism from $(G_2 \times G_3)$ to $\Aut(C_{13})$. Thus, $G$ contains an element of order $6\cdot 13$, which is a contradiction as $G \simeq C_2 \times C_2$. Therefore, $d = 2$, and so $G \simeq (C_7^2 \times C_{13}) \rtimes (G_2 \times G_3)$. Consequently, $(G_2 \times G_3) \leqslant \Aut(C_7^2 \times C_{13})$, $24 \mid |G/C_G(G_7)|$ and $G/C_G(G_7) \leqslant \GL(2,7)$. Using these facts together with GAP \cite{GAP4} computations, we obtain $(G_2 \times G_3) \simeq C_3 \times C_3 \times Q_8,  \ C_3 \times C_3 \times D_8, \ C_3 \times C_6 \times Q_8$ or $ C_3 \times C_6 \times D_8$, which is not possible because $G$ contains an element of order $3 \cdot 4 \cdot 13$ as $G/C_G(G_{13}) \simeq C_6$. Hence, $(z^*)$ is also not realizable by an MP* group as MP* groups are quotient closed.
		
	\end{proof}

\begin{corollary}
The prime graphs realizable by a non-trivial MP group are precisely (a)-(g), (j)-(o) in \Cref{Examplesprime}.
	\end{corollary}

\section{N-PRIME GRAPHS OF SOLVABLE USR GROUPS}	

In this section, we classify the N-prime graphs of metanilpotent USR groups. As mentioned in the introduction, the N-prime graph $\Gamma_N(G)$ of a group $G$ is a simple directed graph whose vertex set is $\pi(G)$ and, for distinct vertices $p$ and $q$, the arc $q \rightarrow p$ (or $p\leftarrow q$) is in the graph $\Gamma_N(G)$ if and only if $G$ contains elements $x$ and $y$ of order $p$ and $q$, respectively, such that $\langle x \rangle$ is normalized by $y$ in $G$.
Clearly, prime graph of $G$ can be obtained from $\Gamma_N(G)$ replacing every double arrow $q \rightleftarrows p$ with the edge $q - p$ and deleting all single arrows; but the converse is not true. There exist examples of groups having the same prime graph but different N-prime graphs \cite[Remark~3.1]{PdRV25}. For example, the $2$-Frobenius group $(C_3^6 \rtimes C_7) \rtimes C_3$ and the Frobenius group $C_3^6 \rtimes C_7$ have the same prime graph $(3~~7)$, but $\Gamma_N((C_3^6 \rtimes C_7) \rtimes C_3)$ is $(3 \rightarrow 7)$ whereas $\Gamma_N(C_3^6 \rtimes C_7)$ is $(3~~7)$. So, the N-prime graph of a group, in principle, encodes more detailed information about the structure of the group.\\

The following lemma shall be useful for upcoming results.

\begin{lemma}\label{Arcs}
Let $G$ be a finite group with the prime graph $\Gamma(G)$ and the N-prime graph $\Gamma_N(G)$. Let $p$ and $q \in \pi(G)$, then 
\begin{itemize}
	\item[(i)]  If $q \rightarrow p \in \Gamma_N(G)$ and $p \rightarrow q \not \in \Gamma_N(G)$, then $q$ divides $p-1$. Also, $p-q\in \Gamma(G)$ if and only if $q \rightarrow p \in \Gamma_N(G)$ and $p \rightarrow q \in \Gamma_N(G)$.

\item[(ii)] Let $P$ be a $p$-subgroup of $G$ which is normalized by an element $y$ of order $q^k$ in $G$. If $q^k$ divides $p-1$ for some positive integer $k$, then $q \rightarrow p \in \Gamma_N(G)$.   
\item[(iii)] If $G$ is a USR group and $2q$ divides $p-1$, then $q \rightarrow p \in \Gamma_N(G)$.
	\end{itemize} 
\end{lemma}
\begin{proof}
(i) and (ii) follow from Remark~3.1 and \cite[ Lemma~ 3.2]{PdRV25}, respectively. For (iii), let $a \in G$ be an element of order $p$, so that $|\operatorname{Aut}\langle a \rangle| = p-1$ and thus $q$ divides $|B_G(a)|$. Consequently, $q$ divides $|N_G(a)|$ and thus the arc $q \rightarrow p \in \Gamma_N(G)$. 

\end{proof}

\begin{figure}[htbp]
	\centering
	\scriptsize
	\setlength{\tabcolsep}{2pt}
	\renewcommand{\arraystretch}{1.15}
	
	\begin{adjustbox}{max width=\textwidth}
		\begin{tabular}{|c|c|}
				\hline
				$\pi(G)$& N-prime graphs\\
				\hline
				
				$\{2\}$ &
				\fl{a} \begin{subfigure}{.15\textwidth}
					\centering
					\begin{tikzpicture}
						\draw (0.5,1.5) node{$C_2$};
						\node[label=east:{$2$}] at (0.5,1) (3){};
						
					\end{tikzpicture}
				\end{subfigure} \\
				\hline
				$\{3\}$&
				\fl{b} \begin{subfigure}{.15\textwidth}
					\centering
					\begin{tikzpicture}
						\draw (0.5,1.5) node{$C_3$};
						\node[label=east:{$3$}] at (0.5,1) (3){};
						
					\end{tikzpicture}
				\end{subfigure} \\
				\hline
				$\{2,3\}$  &	\fl{c}
				\begin{subfigure}{.18\textwidth}
					\centering
					\begin{tikzpicture}
						\draw (0.5,1.5) node{$A_4$};
						\node[label=west:{$2$}] at (0,1) (2){};
						\node[label=east:{$3$}] at (0.5,1) (3){};
						
					\end{tikzpicture}
				\end{subfigure}
				\hspace{1cm}
				\fl{d}
				\begin{subfigure}{.18\textwidth}
					\centering
					\begin{tikzpicture}
						\draw (0.5,1.5) node{$S_3=C_3\rtimes_\Frob C_2$};
						\node[label=west:{$2$}] at (0,1) (2){};
						\node[label=east:{$3$}] at (0.5,1) (3){};
						\draw[->, thick] ($(2)+(-0.15,0.0)$) -- ($(3)+(0.15,0.0)$);
					\end{tikzpicture}
				\end{subfigure}
				\hspace{1cm}		\fl{e}\hspace{-0.4cm} 
				\begin{subfigure}{.18\textwidth}
					\centering
					\begin{tikzpicture}
						\draw (0.5,1.5) node{$S_3 \times C_2$};
						\node[label=west:{$2$}] at (0,1) (2){};
						\node[label=east:{$3$}] at (0.5,1) (3){};
						\draw[->, thick] ($(2)+(-0.15,0.08)$) -- ($(3)+(0.15,0.08)$);		
						\draw[->, thick] ($(3)+(0.15,-0.08)$) -- ($(2)+(-0.15,-0.08)$);		
					\end{tikzpicture}
				\end{subfigure}  \\
				\hline
				$ \{2,5\} $  &	\fl{f}\hspace{1cm} 
				\begin{subfigure}{.18\textwidth}
					\centering
					\begin{tikzpicture}
						\draw (0.5,1.2) node{$C_5^2 \rtimes_\Frob Q_8$};				
						\node[label=west:{$2$}] at (0,0.8) (2){};
						\node[label=west:{$5$}] at (0,0) (5){};
						\draw[->, thick] ($(2)+(-0.3,-0.15)$) -- ($(5)+(-0.3,0.15)$);		
					\end{tikzpicture}
				\end{subfigure}
				\hspace{-0.4cm} 	\fl{g}\hspace{1cm} 
				\begin{subfigure}{.20\textwidth}
					\centering
					
					\begin{tikzpicture}
						\draw (0.5,1.2) node{$[C_5^2 \rtimes_\Frob Q_8] \times C_2$};								
						\node[label=west:{$2$}] at (0,0.8) (2){};
						\node[label=west:{$5$}] at (0,0) (5){};
						\draw[->, thick] ($(2)+(-0.24,-0.15)$) -- ($(5)+(-0.24,0.15)$);		
						\draw[->, thick] ($(5)+(-0.35,0.15)$) -- ($(2)+(-0.35,-0.15)$);		
					\end{tikzpicture}
				\end{subfigure} \\
				\hline
				$\{3, 7\}$ & 
				\fl{h}\hspace{1cm} 
				\begin{subfigure}{.17\textwidth}
					\centering
					\begin{tikzpicture}
						\draw (0.5,1.2) node{$C_7 \rtimes_\Frob C_3$};
						\node[label=west:{$3$}] at (0,0.8) (3){};
						\node[label=west:{$7$}] at (0,0) (7){};
						\draw[->, thick] ($(3)+(-0.3,-0.15)$) -- ($(7)+(-0.3,0.15)$);		
					\end{tikzpicture}
				\end{subfigure}\\
				\hline
				
				$\{2,3,5\}$  & 	\fl{i} \begin{subfigure}{.15\textwidth}
					\centering
					\begin{tikzpicture}
						\draw (0.5,1.5) node{$C_5^4 \rtimes_\Frob (C_3 \times Q_8)$};				
						
						\node[label=west:{$2$}] at (0,0.8) (2){};
						\node[label=east:{$3$}] at (0.5,0.8) (3){};
						\node[label=west:{$5$}] at (0,0) (5){};
						\draw[->, thick] ($(2)+(-0.15,0.08)$) -- ($(3)+(0.15,0.08)$);		
						\draw[->, thick] ($(3)+(0.15,-0.08)$) -- ($(2)+(-0.15,-0.08)$);		
						\draw[->, thick] ($(2)+(-0.3,-0.15)$) -- ($(5)+(-0.3,0.15)$);
						
					\end{tikzpicture}			
				\end{subfigure} 
				
				\hspace{-0.3cm}	\fl{j}\hspace{0.5cm}
				\begin{subfigure}{.22\textwidth}
					\centering
					
					\begin{tikzpicture}
						\draw (0.5,1.5) node{$[C_5^4 \rtimes_\Frob (C_3 \times Q_8)] \times C_2$};								
						\node[label=west:{$2$}] at (0,0.8) (2){};
						\node[label=east:{$3$}] at (0.5,0.8) (3){};
						\node[label=west:{$5$}] at (0,0) (5){};
						
						\draw[->, thick] ($(2)+(-0.15,0.08)$) -- ($(3)+(0.15,0.08)$);		
						\draw[->, thick] ($(3)+(0.15,-0.08)$) -- ($(2)+(-0.15,-0.08)$);		
						
						\draw[->, thick] ($(2)+(-0.24,-0.15)$) -- ($(5)+(-0.24,0.15)$);		
						\draw[->, thick] ($(5)+(-0.35,0.15)$) -- ($(2)+(-0.35,-0.15)$);		
						
					\end{tikzpicture}
				\end{subfigure}
				\hspace{0.3cm} 	\fl{k}\hspace{-0.4cm}
				\begin{subfigure}{.22\textwidth}
					\centering
					\begin{tikzpicture}
						\draw (0.5,1.5) node{$C_{15} \rtimes C_4$};				
						\node[label=west:{$2$}] at (0,0.8) (2){};
						\node[label=east:{$3$}] at (0.5,0.8) (3){};
						\node[label=west:{$5$}] at (0,0) (5){};
						\draw[->, thick] ($(2)+(-0.15,0.08)$) -- ($(3)+(0.15,0.08)$);		
						\draw[->, thick] ($(3)+(0.15,-0.08)$) -- ($(2)+(-0.15,-0.08)$);		
						
						\draw[->, thick] ($(2)+(-0.3,-0.15)$) -- ($(5)+(-0.3,0.15)$);		
						\draw[->, thick] ($(3)+(0.2,-0.12)$) -- ($(5)+(-0.15,0.1)$);
						\draw[->, thick] ($(5)+(-0.1,-0.01)$) -- ($(3)+(0.3,-0.18)$);
						
					\end{tikzpicture}
				\end{subfigure} 
				\hspace{-0.5cm} 	\fl{l}\hspace{-0.2cm}
				\begin{subfigure}{.22\textwidth}
					\centering
					
					\begin{tikzpicture}
						\draw (0.5,1.5) node{$[C_{15} \rtimes C_4] \times S_3$};								
						\node[label=west:{$2$}] at (0,0.8) (2){};
						\node[label=east:{$3$}] at (0.5,0.8) (3){};
						\node[label=west:{$5$}] at (0,0) (5){};
						\draw[->, thick] ($(2)+(-0.15,0.08)$) -- ($(3)+(0.15,0.08)$);		
						\draw[->, thick] ($(3)+(0.15,-0.08)$) -- ($(2)+(-0.15,-0.08)$);		
						
						\draw[->, thick] ($(2)+(-0.24,-0.15)$) -- ($(5)+(-0.24,0.15)$);		
						\draw[->, thick] ($(5)+(-0.36,0.15)$) -- ($(2)+(-0.36,-0.15)$);		
						\draw[->, thick] ($(3)+(0.2,-0.12)$) -- ($(5)+(-0.15,0.1)$);
						\draw[->, thick] ($(5)+(-0.1,-0.01)$) -- ($(3)+(0.3,-0.18)$);
						
					\end{tikzpicture}
				\end{subfigure}\\
				\hline
				$\{2,3,7\}$  & 	\fl{m}\hspace{-0.5cm}
				\begin{subfigure}{.2\textwidth}
					\centering
					\begin{tikzpicture}
						\draw (0.5,1.2) node{$C_{7} \rtimes_\Frob  C_6$};				
						\node[label=west:{$2$}] at (0,0.8) (2){};
						\node[label=east:{$3$}] at (0.5,0.8) (3){};
						\node[label=east:{$7$}] at (0.5,0) (7){};
						\draw[->, thick] ($(2)+(-0.15,0.08)$) -- ($(3)+(0.15,0.08)$);		
						\draw[->, thick] ($(3)+(0.15,-0.08)$) -- ($(2)+(-0.15,-0.08)$);		
						\draw[->, thick] ($(2)+(-0.3,-0.17)$) -- ($(7)+(0.1, 0.1)$);		
						\draw[->, thick] ($(3)+(0.3,-0.2)$) -- ($(7)+(0.3,0.2)$);		
						
					\end{tikzpicture}
				\end{subfigure}
				\fl{n}\hspace{-0.5cm}
				\begin{subfigure}{.22\textwidth}
					\centering
					\begin{tikzpicture}
						\draw (0.5,1.2) node{$[C_{7} \rtimes_\Frob  C_6] \times C_3$};								
						\node[label=west:{$2$}] at (0,0.8) (2){};
						\node[label=east:{$3$}] at (0.5,0.8) (3){};
						\node[label=east:{$7$}] at (0.5,0) (7){};
						\draw[->, thick] ($(2)+(-0.15,0.08)$) -- ($(3)+(0.15,0.08)$);		
						\draw[->, thick] ($(3)+(0.15,-0.08)$) -- ($(2)+(-0.15,-0.08)$);		
						\draw[->, thick] ($(2)+(-0.3,-0.17)$) -- ($(7)+(0.1, 0.1)$);		
						\draw[->, thick] ($(3)+(0.24,-0.2)$) -- ($(7)+(0.24,0.2)$);		
						\draw[->, thick] ($(7)+(0.36,0.2)$) -- ($(3)+(0.36,-0.2)$);
						
					\end{tikzpicture}
				\end{subfigure}
				\hspace{-0.2cm} 	\fl{o}\hspace{-0.5cm}
				\begin{subfigure}{.22\textwidth}
					\centering
					\begin{tikzpicture}
						\draw (0.5,1.2) node{$[C_{7} \rtimes_\Frob  C_6] \times C_2$};				
						\node[label=west:{$2$}] at (0,0.8) (2){};
						\node[label=east:{$3$}] at (0.5,0.8) (3){};
						\node[label=east:{$7$}] at (0.5,0) (7){};
						\draw[->, thick] ($(2)+(-0.15,0.08)$) -- ($(3)+(0.15,0.08)$);		
						\draw[->, thick] ($(3)+(0.15,-0.08)$) -- ($(2)+(-0.15,-0.08)$);		
						\draw[->, thick] ($(7)+(0.1,-0.04)$) -- ($(2)+(-0.4,-0.24)$);
						\draw[->, thick] ($(3)+(0.3,-0.2)$) -- ($(7)+(0.3,0.2)$);				
						\draw[->, thick] ($(2)+(-0.3,-0.17)$) -- ($(7)+(0.15, 0.07)$);				
					\end{tikzpicture}
				\end{subfigure}
				\hspace{-0.2cm} 	\fl{p}\hspace{-0.3cm}
				\begin{subfigure}{.2\textwidth}
					\centering
					\begin{tikzpicture}
						\draw (0.5,1.2) node{$[C_{7} \rtimes_\Frob  C_3] \times S_3$};				
						\node[label=west:{$2$}] at (0,0.8) (2){};
						\node[label=east:{$3$}] at (0.5,0.8) (3){};
						\node[label=east:{$7$}] at (0.5,0) (7){};
						\draw[->, thick] ($(2)+(-0.15,0.08)$) -- ($(3)+(0.15,0.08)$);		
						\draw[->, thick] ($(3)+(0.15,-0.08)$) -- ($(2)+(-0.15,-0.08)$);		
						\draw[->, thick] ($(7)+(0.1,-0.04)$) -- ($(2)+(-0.4,-0.24)$);		
						\draw[->, thick] ($(2)+(-0.3,-0.17)$) -- ($(7)+(0.15, 0.07)$);
						\draw[->, thick] ($(3)+(0.24,-0.2)$) -- ($(7)+(0.24,0.2)$);		
						\draw[->, thick] ($(7)+(0.36,0.2)$) -- ($(3)+(0.36,-0.2)$);
					\end{tikzpicture}
				\end{subfigure}\\
				\hline
				
				$ \{2,3,13\}$  & 	\fl{q}\hspace{-0.5cm}

				\begin{subfigure}{.28\textwidth}
					\centering
					\begin{tikzpicture}
						\draw (0.5,1.2) node{$[C_{13} \rtimes_\Frob C_6]$};				
						\node[label=west:{$2$}] at (0,0.8) (2){};
						\node[label=east:{$3$}] at (0.5, 0.8) (3){};
						
						\node[label=east:{$13$}] at (0.75, 0) (13){};
						\draw[->, thick] ($(2)+(-0.15,0.08)$) -- ($(3)+(0.15,0.08)$);		
						\draw[->, thick] ($(3)+(0.15,-0.08)$) -- ($(2)+(-0.15,-0.08)$);		
						\draw[->, thick] ($(2)+(-0.3,-0.17)$) -- ($(13)+(0.2, 0)$);				
						\draw[->, thick] ($(3)+(0.4,-0.17)$) -- ($(13)+(0.4, 0.17)$);			
						
					\end{tikzpicture}
				\end{subfigure}

				\hspace{-0.2cm}  	\fl{r}\hspace{-0.3cm}
				\begin{subfigure}{.28\textwidth}
					\centering
					\begin{tikzpicture}
						\draw (0.5,1.2) node{$[C_{13} \rtimes_\Frob C_6]  \times C_2$};				
						\node[label=west:{$2$}] at (0,0.8) (2){};
						\node[label=east:{$3$}] at (0.5, 0.8) (3){};
						\node[label=east:{$13$}] at (0.75, 0) (13){};
						\draw[->, thick] ($(2)+(-0.15,0.08)$) -- ($(3)+(0.15,0.08)$);		
						\draw[->, thick] ($(3)+(0.15,-0.08)$) -- ($(2)+(-0.15,-0.08)$);		
						\draw[->, thick] ($(2)+(-0.3,-0.17)$) -- ($(13)+(0.2, 0)$);	
						\draw[->, thick] ($(13)+(0.1,-0.13)$) -- ($(2)+(-0.4, -0.3)$);				
						\draw[->, thick] ($(3)+(0.4,-0.17)$) -- ($(13)+(0.4, 0.17)$);

					\end{tikzpicture}
				\end{subfigure} 
				\hspace{0.4cm}  	\fl{s}\hspace{-0.4cm}
				\begin{subfigure}{.28\textwidth}
					\centering
					\begin{tikzpicture}
						\draw (0.5,1.2) node{$[C_{13} \rtimes_\Frob C_6] \times S_3$};					
						\node[label=west:{$2$}] at (0,0.8) (2){};
						\node[label=east:{$3$}] at (0.5, 0.8) (3){};
						\node[label=east:{$13$}] at (0.75, 0) (13){};
						\draw[->, thick] ($(2)+(-0.15,0.08)$) -- ($(3)+(0.15,0.08)$);		
						\draw[->, thick] ($(3)+(0.15,-0.08)$) -- ($(2)+(-0.15,-0.08)$);		
						\draw[->, thick] ($(2)+(-0.3,-0.17)$) -- ($(13)+(0.2, 0)$);	
						\draw[->, thick] ($(13)+(0.1,-0.13)$) -- ($(2)+(-0.4, -0.3)$);			
						\draw[->, thick] ($(3)+(0.4,-0.17)$) -- ($(13)+(0.4, 0.17)$);		
						\draw[->, thick] ($(13)+(0.55,0.17)$) -- ($(3)+(0.52, -0.1)$);		
						
					\end{tikzpicture}
				\end{subfigure}\\ [0.3cm]

				& 	\fl{x*}\hspace{-0.3cm}	
				\begin{subfigure}{.28\textwidth}
					\centering
					\begin{tikzpicture}
						
						\node[label=west:{$2$}] at (0,0.8) (2){};
						\node[label=east:{$3$}] at (0.5, 0.8) (3){};
						\node[label=east:{$13$}] at (0.75, 0) (13){};
						\draw[->, thick] ($(2)+(-0.15,0.08)$) -- ($(3)+(0.15,0.08)$);		
						\draw[->, thick] ($(3)+(0.15,-0.08)$) -- ($(2)+(-0.15,-0.08)$);		
						\draw[->, thick] ($(2)+(-0.3,-0.17)$) -- ($(13)+(0.2, 0)$);	
						\draw[->, thick] ($(3)+(0.4,-0.17)$) -- ($(13)+(0.4, 0.17)$);		
						\draw[->, thick] ($(13)+(0.55,0.17)$) -- ($(3)+(0.52, -0.1)$);		
						
					\end{tikzpicture}
				\end{subfigure} \\
				\hline

				$\{2,3,5,7\}$  &  	\fl{t}\hspace{-0.3cm}			
				\begin{subfigure}{.28\textwidth}
					\centering
					\begin{tikzpicture}
						\node[align=center, font=\scriptsize, text width=4cm] at (0.35,1.2){$[C_5^2 \rtimes_\Frob Q_8] \times [C_{7} \rtimes_\Frob  C_3]$};				
						\node[label=west:{$2$}] at (0,0.8) (2){};
						\node[label=east:{$3$}] at (0.5, 0.8) (3){};
						\node[label=west:{$5$}] at (0, 0) (5){};
						\node[label=east:{$7$}] at (0.5, 0) (7){};
						\draw[->, thick] ($(2)+(-0.15,0.08)$) -- ($(3)+(0.15,0.08)$);		
						\draw[->, thick] ($(3)+(0.15,-0.04)$) -- ($(2)+(-0.15,-0.04)$);
						\draw[->, thick] ($(2)+(-0.32,-0.15)$) -- ($(5)+(-0.32,0.15)$);
						\draw[->, thick] ($(2)+(-0.2,-0.11)$) -- ($(7)+(0.2, 0.22)$);		
						\draw[->, thick] ($(7)+(0.2,0.1)$) -- ($(2)+(-0.25, -0.2)$);		
						\draw[->, thick] ($(3)+(0.3,-0.2)$) -- ($(7)+(0.3,0.2)$);		
						\draw[->, thick] ($(5)+(-0.15,0.04)$) -- ($(7)+(0.15,0.04)$);		
						\draw[->, thick] ($(7)+(0.15,-0.08)$) -- ($(5)+(-0.15,-0.08)$);		
						\draw[->, thick] ($(3)+(0.2,-0.15)$) -- ($(5)+(-0.2,0.1)$);		
						\draw[->, thick] ($(5)+(-0.2,0.2)$) -- ($(3)+(0.1,-0.1)$);		
					\end{tikzpicture}
				\end{subfigure}
				\hspace{-0.2cm} 	\fl{u}\hspace{-0.3cm}
				\begin{subfigure}{.28\textwidth}
					\centering
					\begin{tikzpicture}
						\node[align=center, font=\scriptsize, text width=4cm] at (0.35,1.2){$[C_5^2 \rtimes_\Frob Q_8] \times [C_{7} \rtimes_\Frob  C_3] \times C_2$};				
						\node[label=west:{$2$}] at (0,0.8) (2){};
						\node[label=east:{$3$}] at (0.5, 0.8) (3){};
						\node[label=west:{$5$}] at (0, 0) (5){};
						\node[label=east:{$7$}] at (0.5, 0) (7){};
						\draw[->, thick] ($(2)+(-0.15,0.08)$) -- ($(3)+(0.15,0.08)$);		
						\draw[->, thick] ($(3)+(0.15,-0.04)$) -- ($(2)+(-0.15,-0.04)$);
						\draw[->, thick] ($(2)+(-0.3,-0.15)$) -- ($(5)+(-0.3,0.15)$);
						\draw[->, thick] ($(5)+(-0.4,0.15)$) -- ($(2)+(-0.4,-0.15)$);
						\draw[->, thick] ($(2)+(-0.2,-0.11)$) -- ($(7)+(0.2, 0.22)$);		
						\draw[->, thick] ($(7)+(0.2,0.1)$) -- ($(2)+(-0.25, -0.2)$);		
						\draw[->, thick] ($(3)+(0.3,-0.2)$) -- ($(7)+(0.3,0.2)$);		
						\draw[->, thick] ($(5)+(-0.15,0.04)$) -- ($(7)+(0.15,0.04)$);		
						\draw[->, thick] ($(7)+(0.15,-0.08)$) -- ($(5)+(-0.15,-0.08)$);		
						\draw[->, thick] ($(3)+(0.2,-0.15)$) -- ($(5)+(-0.2,0.1)$);		
						\draw[->, thick] ($(5)+(-0.2,0.2)$) -- ($(3)+(0.1,-0.1)$);					
					\end{tikzpicture}
				\end{subfigure} 
				\hspace{0.4cm} 	\fl{v}\hspace{-0.4cm}
				\begin{subfigure}{.28\textwidth}
					\centering
					\begin{tikzpicture}
						\node[align=center, font=\scriptsize, text width=4cm] at (0.35,1.2){$[C_5^2 \rtimes_\Frob Q_8] \times [C_{7} \rtimes_\Frob  C_6] \times S_3$};					
						\node[label=west:{$2$}] at (0,0.8) (2){};
						\node[label=east:{$3$}] at (0.5, 0.8) (3){};
						\node[label=west:{$5$}] at (0, 0) (5){};
						\node[label=east:{$7$}] at (0.5, 0) (7){};
						\draw[->, thick] ($(2)+(-0.15,0.08)$) -- ($(3)+(0.15,0.08)$);		
						\draw[->, thick] ($(3)+(0.15,-0.04)$) -- ($(2)+(-0.15,-0.04)$);
						\draw[->, thick] ($(2)+(-0.3,-0.15)$) -- ($(5)+(-0.3,0.15)$);
						\draw[->, thick] ($(5)+(-0.4,0.15)$) -- ($(2)+(-0.4,-0.15)$);
						\draw[->, thick] ($(2)+(-0.2,-0.11)$) -- ($(7)+(0.2, 0.22)$);		
						\draw[->, thick] ($(7)+(0.2,0.1)$) -- ($(2)+(-0.25, -0.2)$);		
						\draw[->, thick] ($(3)+(0.3,-0.2)$) -- ($(7)+(0.3,0.2)$);		
						\draw[->, thick] ($(5)+(-0.15,0.04)$) -- ($(7)+(0.15,0.04)$);		
						\draw[->, thick] ($(7)+(0.15,-0.08)$) -- ($(5)+(-0.15,-0.08)$);		
						\draw[->, thick] ($(3)+(0.2,-0.15)$) -- ($(5)+(-0.2,0.1)$);		
						\draw[->, thick] ($(5)+(-0.2,0.2)$) -- ($(3)+(0.1,-0.1)$);	
						\draw[->, thick] ($(7)+(0.4,0.2)$) -- ($(3)+(0.4,-0.2)$);						
					\end{tikzpicture}
				\end{subfigure}\\
				\hline		
				
				$\{2,3,5,13\}$  &
				\fl{w}\hspace{-0.5cm}
				\begin{subfigure}{.45\textwidth} 
					\centering
					\begin{tikzpicture}
						\node[align=center, font=\scriptsize, text width=4cm] at (0.35,1.2) {$[C_5^2 \rtimes_{\Frob} Q_8] \times [C_{13} \rtimes_{\Frob} C_6]$};				
						\node[label=west:{$2$}] at (0,0.8) (2){};
						\node[label=east:{$3$}] at (0.7, 0.8) (3){};
						\node[label=west:{$5$}] at (0, 0) (5){};
						\node[label=east:{$13$}] at (0.7, 0) (13){};
						\draw[->, thick] ($(2)+(-0.15,0.08)$) -- ($(3)+(0.15,0.08)$);		
						\draw[->, thick] ($(3)+(0.15,-0.04)$) -- ($(2)+(-0.15,-0.04)$);
						\draw[->, thick] ($(2)+(-0.3,-0.15)$) -- ($(5)+(-0.3,0.15)$);
						\draw[->, thick] ($(5)+(-0.4,0.15)$) -- ($(2)+(-0.4,-0.15)$);
						\draw[->, thick] ($(2)+(-0.2,-0.11)$) -- ($(13)+(0.2, 0.22)$);		
						\draw[->, thick] ($(13)+(0.2,0.1)$) -- ($(2)+(-0.25, -0.2)$);		
						\draw[->, thick] ($(3)+(0.32,-0.2)$) -- ($(13)+(0.32,0.2)$);		
						\draw[->, thick] ($(5)+(-0.15,0.04)$) -- ($(13)+(0.15,0.04)$);		
						\draw[->, thick] ($(13)+(0.15,-0.08)$) -- ($(5)+(-0.15,-0.08)$);		
						\draw[->, thick] ($(3)+(0.2,-0.15)$) -- ($(5)+(-0.2,0.1)$);		
						\draw[->, thick] ($(5)+(-0.2,0.2)$) -- ($(3)+(0.1,-0.1)$);	
						
					\end{tikzpicture}
				\end{subfigure}
				\hfill 
				\fl{x}\hspace{-0.3cm}
				\begin{subfigure}{.45\textwidth}
					\centering
					\begin{tikzpicture}
						\node[align=center, font=\scriptsize, text width=4cm] at (0.35,1.2) {$[C_5^2 \rtimes_{\Frob} Q_8] \times [C_{13} \rtimes_{\Frob} C_6] \times S_3$};				
						
						\node[label=west:{$2$}] at (0,0.8) (2){};
						\node[label=east:{$3$}] at (0.7, 0.8) (3){};
						\node[label=west:{$5$}] at (0, 0) (5){};
						\node[label=east:{$13$}] at (0.7, 0) (13){};
						\draw[->, thick] ($(2)+(-0.15,0.08)$) -- ($(3)+(0.15,0.08)$);		
						\draw[->, thick] ($(3)+(0.15,-0.04)$) -- ($(2)+(-0.15,-0.04)$);
						\draw[->, thick] ($(2)+(-0.3,-0.15)$) -- ($(5)+(-0.3,0.15)$);
						\draw[->, thick] ($(5)+(-0.4,0.15)$) -- ($(2)+(-0.4,-0.15)$);
						\draw[->, thick] ($(2)+(-0.2,-0.11)$) -- ($(13)+(0.2, 0.22)$);		
						\draw[->, thick] ($(13)+(0.2,0.1)$) -- ($(2)+(-0.25, -0.2)$);		
						\draw[->, thick] ($(3)+(0.3,-0.2)$) -- ($(13)+(0.3,0.2)$);		
						\draw[->, thick] ($(5)+(-0.15,0.04)$) -- ($(13)+(0.15,0.04)$);		
						\draw[->, thick] ($(13)+(0.15,-0.08)$) -- ($(5)+(-0.15,-0.08)$);		
						\draw[->, thick] ($(3)+(0.2,-0.15)$) -- ($(5)+(-0.2,0.1)$);		
						\draw[->, thick] ($(5)+(-0.2,0.2)$) -- ($(3)+(0.1,-0.1)$);	
						\draw[->, thick] ($(13)+(0.4,0.2)$) -- ($(3)+(0.4,-0.2)$);							
						
					\end{tikzpicture}
				\end{subfigure}\\
				\hline 
				
				$\{2,3,7,13\}$  &
				\fl{y*}\hspace{-0.5cm}
				\begin{subfigure}{.45\textwidth} 
					\centering
					\begin{tikzpicture}
						
						\node[label=west:{$2$}] at (0,0.8) (2){};
						\node[label=east:{$3$}] at (0.7, 0.8) (3){};
						\node[label=west:{$7$}] at (0, 0) (7){};
						\node[label=east:{$13$}] at (0.7, 0) (13){};
						\draw[->, thick] ($(2)+(-0.15,0.08)$) -- ($(3)+(0.15,0.08)$);		
						\draw[->, thick] ($(3)+(0.15,-0.04)$) -- ($(2)+(-0.15,-0.04)$);
						\draw[->, thick] ($(2)+(-0.3,-0.15)$) -- ($(7)+(-0.3,0.15)$);
						\draw[->, thick] ($(7)+(-0.4,0.15)$) -- ($(2)+(-0.4,-0.15)$);
						\draw[->, thick] ($(2)+(-0.2,-0.11)$) -- ($(13)+(0.2, 0.22)$);		
						\draw[->, thick] ($(13)+(0.2,0.1)$) -- ($(2)+(-0.25, -0.2)$);		
						\draw[->, thick] ($(3)+(0.32,-0.2)$) -- ($(13)+(0.32,0.2)$);		
						\draw[->, thick] ($(7)+(-0.15,0.04)$) -- ($(13)+(0.15,0.04)$);		
						\draw[->, thick] ($(13)+(0.15,-0.08)$) -- ($(7)+(-0.15,-0.08)$);		
						\draw[->, thick] ($(3)+(0.2,-0.15)$) -- ($(7)+(-0.2,0.1)$);		
						\draw[->, thick] ($(7)+(-0.2,0.2)$) -- ($(3)+(0.1,-0.1)$);	
						
					\end{tikzpicture}
				\end{subfigure}\\
				\hline
				
				$\{2,3,5,7,13\}$  &
				\fl{z*}\hspace{-0.5cm}
				\begin{subfigure}{.45\textwidth} 
					\centering

					\begin{tikzpicture}[scale=0.8]
						\node[label=west:{$2$}] at (0,1.8) (2){};
						\node[label=east:{$3$}] at (1.7, 1.8) (3){};
						\node[label=west:{$5$}] at (0, 0) (5){};
						\node[label=east:{$7$}] at (1.7, 0) (7){};
						\node[label=east:{$13$}] at (3.4, 0.9) (13){};
						
						
						\draw[->, thick] ($(2)+(-0.15,0.08)$) -- ($(3)+(0.15,0.08)$);		
						\draw[->, thick] ($(3)+(0.15,-0.04)$) -- ($(2)+(-0.15,-0.04)$);
						\draw[->, thick] ($(2)+(-0.3,-0.15)$) -- ($(5)+(-0.3,0.15)$);
						\draw[->, thick] ($(5)+(-0.4,0.15)$) -- ($(2)+(-0.4,-0.15)$);
						\draw[->, thick] ($(2)+(-0.2,-0.11)$) -- ($(7)+(0.2, 0.22)$);		
						\draw[->, thick] ($(7)+(0.2,0.1)$) -- ($(2)+(-0.25, -0.2)$);		
						\draw[->, thick] ($(3)+(0.3,-0.2)$) -- ($(7)+(0.3,0.2)$);		
						\draw[->, thick] ($(5)+(-0.15,0.04)$) -- ($(7)+(0.15,0.04)$);		
						\draw[->, thick] ($(7)+(0.15,-0.08)$) -- ($(5)+(-0.15,-0.08)$);		
						\draw[->, thick] ($(3)+(0.2,-0.15)$) -- ($(5)+(-0.2,0.1)$);		
						\draw[->, thick] ($(5)+(-0.2,0.2)$) -- ($(3)+(0.1,-0.1)$);	
						\draw[->, thick] ($(7)+(0.4,0.2)$) -- ($(3)+(0.4,-0.2)$);		
						\draw[->, thick] ($(7)+(0.5,0.025)$) -- ($(13)+(0.35,-0.15 )$);
						\draw[->, thick] ($(3)+(0.5,0)$) -- ($(13)+(0.35,0.15)$);
						\draw[->, thick] ($(13)+(0.4,0.25)$) -- ($(3)+(0.5,0.125)$);	
						\draw[->, thick] ($(13)+(0.45,-0.2)$) -- ($(7)+(0.5,-0.1 )$);
						
						\draw[->, thick] ($(5)+(0.15,0.25)$) -- ($(13)+(0.2,-0.05)$);
						\draw[->, thick] ($(13)+(0.2,-0.15)$) -- ($(5)+(0.1,0.15)$);	
						
						\draw[->, thick] ($(2)+(0.15,-0.1)$) -- ($(13)+(0.2,0.15)$);
						\draw[->, thick] ($(13)+(0.2,0.05)$) -- ($(2)+(0.1,-0.2)$);											
					\end{tikzpicture}
				\end{subfigure}\\
				\hline
				
			\end{tabular}
		\end{adjustbox}
		\caption{\label{fig:Examples} Possible N-prime graphs of metanilpotent USR groups.}
	\end{figure}

\begin{theorem}\label{N-prime USR}
Let $G$ be a non-trivial finite metanilpotent USR group. Then $\Gamma_N(G)$ is one of the graphs $\fl{a}$ - $\fl{x}$, $\fl{x*}$, $\fl{y*}$ and $\fl{z*}$  in \Cref{fig:Examples}. Moreover, any of these graphs, except possibly $\fl{x*}$, $\fl{y*}$ and $\fl{z*}$ is realizable as the N-prime graph of a metanilpotent USR group.  
\end{theorem}

\begin{proof}
	
\begin{description}
\item[Case I] $\underline{|\pi(G)| =1}.$\\
 Both the N-prime graphs with the single vertex are realized by metanilpotent USR groups. 

\item[Case II] $\underline{|\pi(G)| =2}.$ \\
If $\pi(G) = \{2, 3\}$, then by using \Cref{Arcs}(i), we get that the N-prime graph $(2 \leftarrow 3)$ is not realizable by a metanilpotent USR group. Thus the only possible N-prime graphs with this prime spectrum are $\fl{c}$, $\fl{d}$ and $\fl{e}$, which are realized by the indicated groups in \Cref{fig:Examples}.

If $\pi(G) = \{2, 5\}$, then there exists an element of order $5$ in $G$. Using \Cref{Arcs}(iii), one can conclude that $2 \rightarrow 5$ is an arc of $\Gamma_N(G)$, i.e., the N-prime graphs $(2$ \ $5)$ and $(2 \leftarrow 5)$ are not realizable by metanilpotent USR groups. Thus the only possible N-prime graphs with this prime spectrum are $\fl{f}$ and $\fl{g}$ which are realized by the indicated groups in \Cref{fig:Examples}.
 
 If $\pi(G) = \{3, 7\}$, then proceeding as above, we have that $\fl{h}$ is the only possible N-prime graph in this case.\\
 
 In view of \Cref{Examplesprime}, we observe that if $G$ is a metanilpotent USR group with $|\pi(G)|\geq 3$, then $2-3 \in \Gamma(G)$. Hence, $2\rightleftarrows 3 \in \Gamma_N(G).$ We use this observation to classify the N-prime graphs of metanilpotent USR groups, with $|\pi(G)| \geq 3.$

\item[Case III] $\underline{|\pi(G)| =3}.$

If $\pi(G) = \{2, 3, 5\}$, we already have that $2 \rightleftarrows 3  \in \Gamma_N(G).$ Also, as there exists an element of order $5$ in $G$, \Cref{Arcs}(iii) implies that $2 \rightarrow 5$ is an arc of $\Gamma_N(G)$. Further, since neither $3 \mid (5-1)$ nor $5 \mid (3-1)$, by \Cref{Arcs}(i), a single arc between $3$ and $5$ is not possible. Thus, there remain four possible N-prime graphs, each of which is realizable by metanilpotent USR groups as depicted in $\fl{i}$ - $\fl{l}$ of \Cref{fig:Examples}.

 Suppose $\pi(G) = \{2, 3, 7\}$. Again, $2 \rightleftarrows 3  \in \Gamma_N(G)$ and since $G$ contains an element of order $7$, using \Cref{Arcs}(iii), we have that $3 \rightarrow 7$ is an arc of $\Gamma_N(G)$. Next, we claim that $2 \rightarrow 7$ is an arc of $\Gamma_N(G)$. This leaves us with only four possible N-prime graphs when $\pi(G) = \{2, 3, 7\}$ and each of the possible N-prime graph is indeed realized by the indicated groups in \fl{m} - \fl{p} of \Cref{fig:Examples}. To prove the claim, let us suppose $2 \rightarrow 7\not \in \Gamma_N(G)$,  then by \Cref{Arcs}(i), $2 \leftarrow 7$ is also not an arc of $\Gamma_N(G)$ as $7$ does not divide $2 - 1$. We already have that $3 \rightarrow 7$ is an arc of $\Gamma_N(G)$, and hence consider both possibilities, namely, when $ 3 \leftarrow 7$ is an arc or not an arc of $\Gamma_N(G)$. If $ 3 \leftarrow 7$ is not an arc of $\Gamma_N(G)$, then $G$ has disconnected prime graph, and hence $G$ is a Frobenius group with N-prime graph $2 \rightleftarrows 3 \rightarrow 7$. Thus, by \cite[main theorem]{PV26}, $G$ is one of the following:
\[
C_7^n \rtimes_{Fr} C_6, \quad
C_7^2 \rtimes_{Fr} \operatorname{SL}_2(3), \quad
C_7^{2n} \rtimes_{Fr} (C_3 \times Q_8), \quad C_7^2 \rtimes_{Fr} \operatorname{SL}_2(3) \cdot C_2
\] 
where $n \geq 1$. Therefore, $G$ contains an element of order $2$ which normalizes Sylow $7$-subgroup of $G$ and using \Cref{Arcs}(ii) we get that $2 \rightarrow 7 \in \Gamma_N(G)$, a contradiction. Now, if $3 \leftarrow 7$ is an arc of $\Gamma_N(G)$, i.e. $3 \rightarrow 7$ and $3 \leftarrow 7$ both are arcs of $\Gamma_N(G)$, then $2 \rightarrow 7$ is also an arc of $\Gamma_N(G)$ as $G$ contains an element $a$ of order $21$ which implies $2$ divides $|N_G(a)|$, a contradiction. Therefore, in both cases, we get contradiction and hence $2 \rightarrow 7 \in \Gamma_N(G)$, proving the claim.
 
If $\pi(G) = \{2, 3, 13\}$, then $G$ contains an element of order $13$ and thus by  \Cref{Arcs}(iii), both the arcs $2 \rightarrow 13$ and $3 \rightarrow 13 \in \Gamma_N(G)$. Also, we already have $2 \leftrightarrows 3  \in \Gamma_N(G)$ and thus the only possible N-prime graphs are $\fl{q}$ - $\fl{s}$ and $\fl{x*}$. These are realized by the indicated groups in \Cref{fig:Examples}, except the graph 
$\fl{x*}$, which remains undecided.

\item[Case IV] $\underline{|\pi(G)| =4 \ \text{or} \ 5}.$\\
If $\pi(G) = \{2, 3, 5, 7\}$, by  Subcase I of Case II of \Cref{Metanilpotent}, we have that $2 \rightleftarrows 3$, $2 \rightleftarrows 7$, $3 \rightleftarrows 5$ and $5\rightleftarrows 7  \in \Gamma_N(G)$. Also, by \Cref{Arcs}(iii), we get that $3 \rightarrow 7$ and $2 \rightarrow 5 \in \Gamma_N(G).$ There remain four possible N-prime graphs with this prime spectrum. Three of them, namely $\fl{t}$, $\fl{u}$ and $\fl{v}$, are realizable by metanilpotent USR groups, as listed in  \Cref{fig:Examples} and the prime graph of remaining one N-prime graph is same as the prime graph of group which was discarded in Subcase I of Case II of \Cref{Metanilpotent}.  

 In similar way, we obtain that if $\pi(G) = \{2, 3, 5, 13\}, \{2, 3, 7, 13\}$ or $\{2, 3, 5, 7, 13\}$, then the possible N-prime graphs are either $\fl{w}$ and $\fl{x}$, which are realized by the indicated groups in \Cref{fig:Examples} or one of  $\fl{y*}$ and $\fl{z*}$, which remain undecided.

\end{description}
\end{proof}

We note that the N-prime graphs left undecided here are precisely those whose corresponding prime graphs were left undecided in \Cref{Metanilpotent}. And as we pointed out in $\Cref{Remain}$, we can observe that the N-prime graphs $\fl{y*}$ and $\fl{z*}$ are realized or not realized simultaneously and hence we can say, we are left with two graphs which remain undecided. 
\begin{center}
	\begin{tabular}{|c|c|c|c|}
	\hline
	&Class $C$ of groups& Directed graph $\bm{\Gamma_N}$& $G \in C$ such that $\Gamma_N(G) = \bm{\Gamma_N} $\\
	
	\hline		
	
	\multirow{2}{*}{\textbf{(I)}}
	& \multirow{2}{*}{Frobenius}
	& \fl{c}, \fl{d}, \fl{f}, \fl{h}, \fl{i}, \fl{m}, \fl{q}& as in  \Cref{fig:Examples}\\
	& & $\Gamma_N^1 = (5\leftrightarrows2 \rightarrow 3)$ &$C_3^4 \rtimes_{Fr} \langle x, y~ |~ x^5 = y^4 = 1, x^y = x^{-1} \rangle$ \\
	\hline
	
	\multirow{6}{*}{	\textbf{(II)}}
	& \multirow{6}{*}{$2$-Frobenius}
	& \fl{d}& $C_2^2 \rtimes (C_3 \rtimes_{Fr} C_2)$\\
	
	& & \fl{f} & $C_2^4 \rtimes (C_5 \rtimes_{Fr} C_4)$\\
	
	&  & \fl{h} & $C_3^6 \rtimes (C_7 \rtimes_{Fr}C_3)$\\
	
	&  & \fl{i} & $C_3^4 \rtimes (C_5 \rtimes_{Fr} C_2)$\\
	
	&  & \fl{m} & $C_2^6 \rtimes (C_7 \rtimes_{Fr} C_6)$\\
	
	&  & \fl{q} & $C_2^{12} \rtimes (C_{13} \rtimes_\Frob C_6)$\\
	&   & $\Gamma_N^2 = (2 \leftrightarrows 3 \rightarrow 7)$ & $C_2^{6} \rtimes (C_{7} \rtimes_\Frob C_3)$\\
	\hline
	\multirow{5}{*}{	\textbf{(III)} }
	& \multirow{5}{*}{Nilpotent-by-abelian}
	& \fl{a}- \fl{e}, \fl{h}, \fl{k}-\fl{s} & as in \Cref{fig:Examples}\\
	&& \fl{f}& $C_5 \rtimes_{Fr} C_4$\\
	& & \fl{g} & $(C_5 \rtimes_{Fr} C_4) \times C_2$ \\
	
	&& \fl{i}& $C_5^2 \rtimes_{Fr} C_6$\\
	& & \fl{j} & $(C_5^2 \rtimes_{Fr} C_6) \times C_2$   \\
	\hline
	\textbf{(IV)} &Metabelian  & \fl{a}-\fl{s} & as in 	\textbf{(III)}  \\
	\hline
	\multirow{2}{*}{	\textbf{(V)}}
	& \multirow{2}{*}{Abelian-by-cyclic}
	& \fl{a}-\fl{k}, \fl{m}-\fl{r} & as in \textbf{(III)} \\
	& & \fl{l} & $(C_{15} \rtimes C_4) \times C_2 =$ SG$[60,7] \times C_2$\\
	\hline
	
	\multirow{2}{*}{\textbf{(VI)}}
	& \multirow{2}{*}{Cyclic-by-abelian}
	& \fl{a}, \fl{b}, \fl{d}-\fl{h}, \fl{k}, \fl{m}-\fl{s} & as in 	\textbf{(III)}\\
	&& \fl{l} & as in \textbf{(V)} \\
	\hline
	\textbf{(VII)}&Metacyclic  &  \fl{a}, \fl{b}, \fl{d}-\fl{h}, \fl{k}-\fl{r} & as in 	\textbf{(III)}\\
	\hline
	\textbf{(VIII)}&MP*  &  \fl{a} - \fl{h},  \fl{k} - \fl{s}& as in \textbf{(III)}\\
	\hline
	
\end{tabular}
\captionof{table}{Groups realizing N-prime graphs}
\label{N-primegraphssubclasses}
\end{center}

\begin{proposition}
If $G$ is a finite non-trivial USR group, then the following statements hold. 
\begin{itemize}
	
\item[(i)]  If $G$ is Frobenius, then $\Gamma_N(G)\in \{ \fl{c}, \fl{d}, \fl{f}, \fl{h}, \fl{i}, \fl{m}, \fl{q}, \Gamma_N^1  \}$.
\item[(ii)] If $G$ is $2$-Frobenius, then $\Gamma_N(G)\in \{  \fl{d}, \fl{f}, \fl{h}, \fl{i}, \fl{m}, \fl{q}, \Gamma_N^2 \}$.
\item[(iii)] If $G$ is nilpotent-by-abelian, then $\Gamma_N(G)\in \{  \fl{a} - \fl{s} \}$.
\item[(iv)] If $G$ is metabelian, then $\Gamma_N(G)\in \{  \fl{a} - \fl{s} \}$.
\item[(v)] If $G$ is abelian-by-cyclic, then $\Gamma_N(G)\in \{  \fl{a} - \fl{r} \}$.
\item[(vi)] If $G$ is cyclic-by-abelian, then $\Gamma_N(G)\in \{  \fl{a}, \fl{b}, \fl{d} - \fl{h}, \fl{k} - \fl{s} \}$.
\item[(vii)] If $G$ is metacyclic, then $\Gamma_N(G)\in \{ \fl{a}, \fl{b}, \fl{d} - \fl{h}, \fl{k} - \fl{r}  \}$.
\item[(vii)] If $G$ is MP*, then $\Gamma_N(G)\in \{ \fl{a} - \fl{h}, \fl{k} - \fl{s}  \}$.

\end{itemize}
Here,  $\Gamma_N^1 = (5 \leftrightarrows 2 \rightarrow 3)$, $\Gamma_N^2 = (2 \leftrightarrows 3 \rightarrow 7)$ and the graphs $\fl{a} - \fl{s}$ are depicted in \Cref{fig:Examples}.  

Moreover, every graph listed in the respective sets are realized as the N-prime graph of a  group belonging to the corresponding class.

\end{proposition}

\begin{proof}

	In (i)-(vii), the graphs mentioned in the sets are realized by respective subclasses of finite USR groups as depicted in \Cref{N-primegraphssubclasses}. We eliminate the remaining possibilities.

(i) follows from \cite[Main theorem]{PV26} and \Cref{Arcs}(ii).  (iii) and (iv) follow from \Cref{nilpotent-by-abelian}, \Cref{N-prime USR}, \Cref{N-primegraphssubclasses} and \Cref{Arcs}(i).  (v) follows from \Cref{abelian-by-cyclic}, \Cref{N-prime USR}, \Cref{N-primegraphssubclasses} and \Cref{Arcs}(i). 
Also, \Cref{cyclic-by-abelian} and \Cref{metacyclic} respectively imply (vi) and (vii) along with \Cref{Arcs}(i), \Cref{N-prime USR}, \Cref{N-primegraphssubclasses} and the fact that for any metacylic group or cyclic-by-abelian group $G$, $2 \rightarrow 3 \in \Gamma_N(G)$ and hence $\fl{c} = (2 ~ ~3)$ is not realizable by these classes.
 (viii) follows from \Cref{MP*}, \Cref{N-prime USR}, \Cref{N-primegraphssubclasses} and \Cref{Arcs}(i). So, need to prove (ii).

For (ii), using \Cref{2-frobenius}, we have that prime graphs realizable by $2$-Frobenius USR groups are (c), (e), (g), (h), (l) and (p).  

If $|\pi(G)| = 2$, then, by \Cref{2-frobenius} along with (i) and (iii) of \Cref{Arcs}, we have that the possible N-prime graphs of $2$-Frobenius USR groups are \fl{c}, \fl{d}, \fl{f} and \fl{h}. Out of these \fl{d}, \fl{f} and \fl{h} are realized by $2$-Frobenius USR groups as specified in \Cref{N-primegraphssubclasses}. Now, let us suppose that there exists a $2$-Frobenius group $G$ such that $\Gamma_N(G) = (2~~3)$. Using \cite[Proposition~3.3]{PdRV25}, we get that  $\Gamma_N(G/F(G)) = (2~~3)$. Since $G/F(G)$ is a Frobenius group, then $G/F(G)$ is one of the groups in $2(a)$ of \cite[Main theorem]{PV26}, thus $F(G)$ is a $3$-group which implies that there exists a $3$-group normalized by an element of order $2$. Using \Cref{Arcs}(ii), we get that $2 \rightarrow 3 \in \Gamma_N(G)$, a contradiction.

If $|\pi(G)| = 3$,  then again using \Cref{2-frobenius} along with (i) and (iii) of \Cref{Arcs}, we have that the possible N-prime graphs of $2$-Frobenius USR groups are \fl{i}, \fl{m}, \fl{q} and $\Gamma_N^2$. All these graphs are realized by $2$-Frobenius USR group as specified in \Cref{N-primegraphssubclasses}. $\Gamma_N^2$ is realized by the $2$-Frobenius USR group $C_2^6 \rtimes (C_7 \rtimes_{Fr} C_3)$, whose upper Frobenius group has matrix realization given by reducing  
\[
 A=
\begin{pmatrix}
	0 & 0 & 0 & 0 & 0 & -1 \\
	1 & 0 & 0 & 0 & 0 & -1 \\
	0 & 1 & 0 & 0 & 0 & -1 \\
	0 & 0 & 1 & 0 & 0 & -1 \\
	0 & 0 & 0 & 1 & 0 & -1 \\
	0 & 0 & 0 & 0 & 1 & -1 \\
	
\end{pmatrix},
\quad
B= \begin{pmatrix}
0 & 0 & 0 & 1 & 0 & 0 \\
1 & 0 & 0 & 0 & 0 & 0 \\
0 & 0 & 0 & 0 & 1 & 0 \\
0 & 1 & 0 & 0 & 0 & 0 \\
0 & 0 & 0 & 0 & 0 & 1 \\
0 & 0 & 1 & 0 & 0 & 0 \\

\end{pmatrix}
.
\]
modulo $2$ and proceeding as in Subcase II of Case II in \Cref{2-frobenius}.
\end{proof}

\begin{remark}
	One can observe from \Cref{N-prime USR} that if prime graphs of two metanilpotent USR groups are same, then their N-prime graphs will also be same except when prime graph is $(2~~3)$.
\end{remark}

 \section{The Prime Graph Question for USR groups}\label{SectionPQ}
 
One of the important problems in the study of integral group rings has been First Zassenhaus problem, originally posed as conjecture which was answered negatively in \cite{EM18}. As an approximation to the first  Zassenhaus conjecture, Kimmerle  \cite{Kim06} had posed the question, often called as the prime graph question (PQ), namely, for a finite group $G$, do the prime graphs of $G$ and that of group of normalized units of its integral group ring coincide. (PQ) is answered affirmatively for all finite solvable and Frobenius groups. For non-solvable groups, a reduction to almost simple groups effectively yields positive answers for many cases, as can be found in \cite{KK17}. In this section, we attempt (PQ) for USR groups using the reduction result which states that the (PQ) has a positive answer for a group $G$, if it has a positive answer for all almost simple images of $G$ (see \cite[Theorem~2.1]{KK17}). Recall that a group $G$ is called almost simple if it is sandwiched between a non-abelian simple group and its automorphism group, i.e., there is a non-abelian simple group $S$ such that $S \simeq \Inn(S) \leqslant G \leqslant \Aut(S)$, where $\Inn(S)$ denotes the group of inner automorphisms of $S$. In this case, $S$ is called the socle of $G$. Clearly, each non-abelian simple group $S$ deﬁnes its own family of almost simple groups parametrized by the conjugacy classes of subgroups of $\Out(S) = \Aut(S)/\Inn(S)$.

In \cite{BKMdR24}, the prime graph question was answered in positive for finite rational groups and for most inverse semi-rational groups. In \cite{Tre17}, Trefethen gave a list of all non-abelian composition factors of quadratic rational groups which contains class of USR groups. Using this list and reduction theorem of Kimmerle and Konovalov, we prove the following result:
\begin{proposition} Let $G$ be a finite USR group such that there is no epimorphism $G \rightarrow H$, where $H$ is an almost simple group whose socle is $M$, $F_4(2)$ or ${}^2E_6(2)$. Then, (PQ) has a positive answer for $G$.  
\end{proposition}

\begin{proof} Since images of USR groups are again USR, it suffices to establish a positive answer for all almost simple USR groups whose socle is not $M$,  $F_4(2)$ or ${}^2E_6(2)$. Going by the arguments in the proof of \cite[Theorem~G]{BKMdR24}, one can prove that the (PQ) has a positive answer for all almost simple USR groups with an alternating socle or sporadic socle other than the simple monster group $M$. It remains to consider almost simple USR groups with a simple composition factor of Lie type and for that we go through the list  of the $32$ composition factors of Lie type in \cite{Tre17}. 
Using \cite[Theorem~A, Theorem~B]{BM17a}, we get a positive answer to the Prime Graph Question for all almost simple groups whose socles are $L_2(7), L_2(11), L_4(3), U_3(5), U_3(8), U_4(3), U_5(2), S_4(4), S_6(2), O_8^{+}(2), G_2(3)$ and ${}^3D_4(2)$. Likewise, $L_2(16), L_3(4), U_3(4)$ and ${}^2F_4(2)'$ yield positive answer in view of \cite[Theorem~A]{BM19}. By \cite[Corollary~1.3, Corollary~1.4, Lemma~6.8]{CM21}, we have that $U_4(2), U_5(4), U_6(2),  S_6(3), S_8(2), O_7(3), \linebreak O_8^+(3), O_8^-(2), O_{10}^-(2)$ and  $G_2(4)$ have a positive answer to the (PQ). Also, $L_2(8), L_2(27), L_3(3)$ and  $U_3(3)$ yield positive answer in view of \cite[Theorem~1]{Gil15}, \cite[Theorem~1.2]{EM25}, \cite[Corollary~4.2]{BK11} and \cite[Theorem~3.1]{KK17}, respectively.
\end{proof}

\newcommand{\etalchar}[1]{$^{#1}$}
\providecommand{\bysame}{\leavevmode\hbox to3em{\hrulefill}\thinspace}
\providecommand{\MR}{\relax\ifhmode\unskip\space\fi MR }
\providecommand{\MRhref}[2]{%
	\href{http://www.ams.org/mathscinet-getitem?mr=#1}{#2}
}
\providecommand{\href}[2]{#2}

\end{document}